\documentclass[english,envcountsect,envcountsame]{svjour3}
\usepackage[T1]{fontenc}
\usepackage[latin9]{inputenc}
\usepackage{babel}
\usepackage{refstyle}
\usepackage{float}
\usepackage{url}
\usepackage{amsmath}
\usepackage{amssymb}
\usepackage{graphicx}
\usepackage[numbers]{natbib}
\usepackage[unicode=true]
 {hyperref}

\makeatletter

\AtBeginDocument{\providecommand\exaref[1]{\ref{exa:#1}}}
\AtBeginDocument{\providecommand\defref[1]{\ref{def:#1}}}
\AtBeginDocument{\providecommand\thmref[1]{\ref{thm:#1}}}
\AtBeginDocument{\providecommand\secref[1]{\ref{sec:#1}}}
\AtBeginDocument{\providecommand\tabref[1]{\ref{tab:#1}}}
\AtBeginDocument{\providecommand\propref[1]{\ref{prop:#1}}}
\AtBeginDocument{\providecommand\appref[1]{\ref{app:#1}}}
\AtBeginDocument{\providecommand\algref[1]{\ref{alg:#1}}}
\AtBeginDocument{\providecommand\asmref[1]{\ref{asm:#1}}}
\AtBeginDocument{\providecommand\corref[1]{\ref{cor:#1}}}
\AtBeginDocument{\providecommand\lemref[1]{\ref{lem:#1}}}
\AtBeginDocument{\providecommand\figref[1]{\ref{fig:#1}}}
\providecommand{\tabularnewline}{\\}
\floatstyle{ruled}
\newfloat{algorithm}{tbp}{loa}
\providecommand{\algorithmname}{Algorithm}
\floatname{algorithm}{\protect\algorithmname}
\RS@ifundefined{subsecref}
  {\newref{subsec}{name = \RSsectxt}}
  {}
\RS@ifundefined{thmref}
  {\def\RSthmtxt{theorem~}\newref{thm}{name = \RSthmtxt}}
  {}
\RS@ifundefined{lemref}
  {\def\RSlemtxt{lemma~}\newref{lem}{name = \RSlemtxt}}
  {}

\usepackage{graphicx}
\usepackage[font=scriptsize]{caption}

\newtheorem{assume}{Assumption}

\newref{lem}{name=Lemma~, Name=Lemma~}
\newref{prop}{name=Proposition~, Name=Proposition~}
\newref{thm}{name=Theorem~, Name=Theorem~}
\newref{cor}{name=Corollary~, Name=Corollary~}
\newref{alg}{name=Algorithm~, Name=Algorithm~}
\newref{def}{name=Definition~, Name=Definition~}
\newref{asm}{name=Assumption~, Name=Assumption~}
\newref{sec}{name=Section~, Name=Section~}
\newref{app}{name=Appendix~, Name=Appendix~}
\newref{fig}{name=Figure~, Name=Figure~}
\newref{tab}{name=Table~, Name=Table~}
\newref{exa}{name=Example~, Name=Example~}

\usepackage{matlab-prettifier}
\usepackage{listings} 
\@ifundefined{showcaptionsetup}{}{%
 \PassOptionsToPackage{caption=false}{subfig}}
\usepackage{subfig}
\makeatother

\begin{document}
\title{Complexity of Chordal Conversion for \\
Sparse Semidefinite Programs with Small Treewidth\thanks{Financial support for this work was provided in part by the NSF CAREER
Award ECCS-2047462 and in part by C3.ai Inc. and the Microsoft Corporation
via the C3.ai Digital Transformation Institute.}}
\titlerunning{Complexity of Chordal Conversion for Sparse SDPs with Small Treewidth}
\author{Richard Y. Zhang }
\institute{R. Y. Zhang \at University of Illinois at Urbana-Champaign\\
\email{ryz@illinois.edu}}
\date{~}
\maketitle
\begin{abstract}
If a sparse semidefinite program (SDP), specified over $n\times n$
matrices and subject to $m$ linear constraints, has an aggregate
sparsity graph $G$ with small treewidth, then chordal conversion
will sometimes allow an interior-point method to solve the SDP in
just $O(m+n)$ time per-iteration, which is a significant speedup
over the $\Omega(n^{3})$ time per-iteration for a direct application
of the interior-point method.  Unfortunately, the speedup is not
guaranteed by an $O(1)$ treewidth in $G$ that is independent of
$m$ and $n$, as a diagonal SDP would have treewidth zero but can
still necessitate up to $\Omega(n^{3})$ time per-iteration. Instead,
we construct an extended aggregate sparsity graph $\overline{G}\supseteq G$
by forcing each constraint matrix $A_{i}$ to be its own clique in
$G$. We prove that a small treewidth in $\overline{G}$ does indeed
guarantee that chordal conversion will solve the SDP in $O(m+n)$
time per-iteration, to $\epsilon$-accuracy in at most $O(\sqrt{m+n}\log(1/\epsilon))$
iterations. This sufficient condition covers many successful applications
of chordal conversion, including the MAX-$k$-CUT relaxation, the
Lovász theta problem, sensor network localization, polynomial optimization,
and the AC optimal power flow relaxation, thus allowing theory to
match practical experience.
\end{abstract}

\global\long\def\one{\mathbf{1}}%
\global\long\def\A{\mathbf{A}}%
\global\long\def\AA{\mathcal{A}}%
\global\long\def\b{\mathbf{b}}%
\global\long\def\c{\mathbf{c}}%
\global\long\def\C{\mathbb{C}}%
\global\long\def\D{\mathbf{D}}%
\global\long\def\E{\mathcal{E}}%
\global\long\def\G{\mathcal{G}}%
\global\long\def\H{\mathbb{H}}%
\global\long\def\Hess{\mathbf{H}}%
\global\long\def\i{\mathbf{i}}%
\global\long\def\J{\mathcal{J}}%
\global\long\def\K{\mathcal{K}}%
\global\long\def\L{\mathbf{L}}%
\global\long\def\N{\mathbb{N}}%
\global\long\def\P{\mathbf{P}}%
\global\long\def\PP{\mathcal{P}}%
\global\long\def\R{\mathbb{R}}%
\global\long\def\S{\mathbb{S}}%
\global\long\def\V{\mathcal{V}}%
\global\long\def\Tim{\mathsf{T}}%
\global\long\def\Mem{\mathsf{M}}%

\global\long\def\Re{\mathrm{Re}\,}%
\global\long\def\Im{\mathrm{Im}\,}%

\global\long\def\inner#1#2{\left\langle #1,#2\right\rangle }%

\global\long\def\o#1{\overline{#1}}%

\global\long\def\rank{\operatorname{rank}}%

\global\long\def\diag{\operatorname{diag}}%

\global\long\def\tr{\operatorname{tr}}%

\global\long\def\vec{\operatorname{vec}}%

\global\long\def\svec{\operatorname{svec}}%
\global\long\def\hvec{\operatorname{hvec}}%

\global\long\def\supp{\operatorname{supp}}%
\global\long\def\smat{\operatorname{smat}}%

\global\long\def\mat{\operatorname{mat}}%

\global\long\def\chol{\operatorname{chol}}%

\global\long\def\IPM{\operatorname{ipm}}%

\global\long\def\interior{\operatorname{Int}}%

\global\long\def\nnz{\operatorname{nnz}}%

\global\long\def\idx{\operatorname{idx}}%
\global\long\def\p{\operatorname{p}}%

\global\long\def\col{\operatorname{col}}%

\global\long\def\wid{\operatorname{wid}}%

\global\long\def\tw{\operatorname{tw}}%

\global\long\def\spar{\operatorname{spar}}%

\global\long\def\clique{\operatorname{clique}}%

\global\long\def\proj{\operatorname{proj}}%

\global\long\def\opt{\mathrm{opt}}%
\global\long\def\ub{\mathrm{ub}}%
\global\long\def\lb{\mathrm{lb}}%
\global\long\def\blk{\mathrm{blk}}%
\global\long\def\V{\mathcal{V}}%

\section{Introduction}

Consider directly applying a general-purpose interior-point method
solver, like SeDuMi~\citep{sturm1999using}, SDPT3~\citep{toh1999sdpt3},
and MOSEK~\citep{mosek2019}, to solve the standard-form semidefinite
program to high accuracy:
\begin{equation}
\min_{X\in\S^{n}}\;\inner CX\quad\text{s.t.}\;X\succeq0,\quad\inner{A_{i}}X\le b_{i}\;\text{for all }i\in\{1,2,\dots,m\}.\tag{SDP}\label{eq:sdp}
\end{equation}
Here, $\S^{n}$ denotes the set of $n\times n$ real symmetric matrices
with inner product $\inner{A_{i}}X=\tr(A_{i}X)$, and $X\succeq0$
means that $X$ is positive semidefinite. At each iteration, the $n\times n$
matrix variable $X$ is generally fully dense, even when problem data
$C,A_{1},\dots,A_{m}\in\S^{n}$ and $b_{1},\dots,b_{m}\in\R$ are
sparse. The per-iteration cost of the solver is usually at least $\Omega(n^{3})$
time, which in practice limits the value of $n$ to no more than a
few thousand.

Instead, to handle $n$ as large as a hundred thousand, researchers
have found empirical success by first performing a simple  preprocessing
step called \emph{chordal conversion} (CC), which was first introduced
in 2001 by \citet*{fukuda2001exploiting}. Suppose that every $S=C-\sum_{i}y_{i}A_{i}\succ0$
factors as $S=LL^{T}$ into a lower-triangular Cholesky factor $L$
that is sparse. It turns out, by defining $J_{j}=\{i:L[i,j]\ne0\}$
as the possible locations of nonzeros in the $j$-th column of $L$,
that (\ref{eq:sdp}) is exactly equivalent\footnote{The equivalence is due to \citet*[Theorem~7]{grone1984positive};
see also \citep[Theorem~2.3]{fukuda2001exploiting} and \citep[Theorem~10.1]{vandenberghe2015chordal}.} to the following
\begin{equation}
\min_{X\in\S^{n}}\quad\inner CX\quad\text{s.t.}\quad\begin{array}{c}
\inner{A_{i}}X\le b_{i}\quad\text{for all }i\in\{1,2,\dots,m\},\\
X[J_{j},J_{j}]\succeq0\quad\text{for all }j\in\{1,2,\dots,n\}.
\end{array}\tag{CC}\label{eq:cc}
\end{equation}
While not immediately obvious, (\ref{eq:cc}) is actually an optimization
over a \emph{sparse} matrix variable $X$, because the matrix elements
that are not indexed by a constraint $\inner{A_{i}}X$ or $X[J_{j},J_{j}]$
can be set to zero without affecting the optimization. Hence, the
point of the reformulation is to reduce the number of optimization
variables, from $\frac{1}{2}n(n+1)$ in (\ref{eq:sdp}) to at most
$\omega\cdot n$ in (\ref{eq:cc}), where $\omega=\max_{j}|J_{j}|$
is defined as the maximum number of nonzeros per column of the Cholesky
factor $L$. 

Clearly, chordal conversion needs $\omega\ll n$ in order to be efficient.
Where this condition holds, the interior-point method solver is consistently
able to solve (\ref{eq:cc}) in just $O(m+n)$ time per-iteration,
thereby solving some of the largest instances of (\ref{eq:sdp}) ever
considered. Unfortunately, a small $\omega$ does not actually guarantee
the empirically observed $O(m+n)$ time figure. Consider the following
counterexample, which has $\omega=1$, but would nevertheless incur
a cost of at least $\Omega(n^{3})$ time per-iteration.
\begin{example}[Diagonal SDP]
\label{exa:diag}Given vectors $c,a_{1},\dots,a_{m}\in\R^{n}$, consider
the following instance of (\ref{eq:sdp}):
\[
\min_{X\in\S^{n}}\;\inner{\diag(c)}X\quad\text{s.t.}\;X\succeq0,\quad\inner{\diag(a_{i})}X\le b_{i}\;\text{for all }i\in\{1,2,\dots,m\}.
\]
One can verify that $J_{j}=\{i:L[i,j]\ne0\}=\{j\}$, so the corresponding
(\ref{eq:cc}) is a linear program over $n$ variables and $m$ constraints:
\[
\min_{x\in\R^{n}}\;\inner cx\quad\text{s.t.}\;x\ge0,\quad\inner{a_{i}}x\le b_{i}\;\text{for all }i\in\{1,2,\dots,m\}.
\]
Despite $\omega=\max_{j}|J_{j}|=1$, if the number of linear constraints
is at least $m=\Omega(n)$, then it would take \emph{any} method at
least $\Omega(n^{3})$ time to take a single iteration. \qed
\end{example}
While a small $\omega\ll n$ is clearly necessary for chordal conversion
to be fast, \exaref{diag} shows that it is not sufficient. In this
paper, we fill this gap, by providing a sufficient condition for a
general-purpose interior-point method to solve (\ref{eq:cc}) in $O(m+n)$
time per-iteration. The sufficient condition covers many successful
applications of chordal conversion, including the MAX-$k$-CUT relaxation~\citep{fukuda2001exploiting,nakata2003exploiting,kim2011exploiting},
the Lovász theta problem~\citep{fukuda2001exploiting,nakata2003exploiting,kim2011exploiting},
sensor network localization~\citep{kim2009exploiting,kim2011exploiting},
polynomial optimization~\citep{parrilo2000structured,lasserre2001global,waki2006sums},
and the AC optimal power flow relaxation in electric grid optimization~\citep{jabr2012exploiting,madani2016promises,Eltved2019,molzahn2019survey},
thus allowing theory to match practical experience.

\subsection{Our results: Complexity of chordal conversion}

In order for chordal conversion to be fast, a well-known \emph{necessary}
condition is for the underlying aggregate sparsity graph, which is
defined $G=(V,E)$ with
\begin{gather}
V=\{1,2,\dots,n\},\quad E=\spar(C)\cup\spar(A_{1})\cup\cdots\cup\spar(A_{m}),\label{eq:Edef}\\
\text{where }\spar(C)\equiv\{(i,j):C[i,j]\ne0\text{ for }i>j\},\nonumber 
\end{gather}
to be ``tree-like'' with a small\emph{ treewidth} $\tw(G)\ll n$.
We defer a formal definition of the treewidth to \defref{tw}, and
only note that the value of $\omega$ is lower-bounded as $\omega\ge1+\tw(G)$.
In other words, while $\omega$ can be decreased by symmetrically
reordering the columns and rows of the data matrices, as in $C\gets\Pi C\Pi^{T}$
and $A_{i}\gets\Pi A_{i}\Pi^{T}$ for some permutation matrix $\Pi$,
actually achieving $\omega\ll n$ is possible only if $\tw(G)\ll n$.
However, as illustrated by \exaref{diag}, even on a graph with $\tw(G)=0$,
and even when using the optimal $\omega=1$, chordal conversion may
still not be fast.

In this paper, we show that a \emph{sufficient} condition for chordal
conversion to be fast is for a supergraph $\o G\supseteq G$, that
additionally captures the correlation between constraints, to also\footnote{Note that $\tw(\o G)\ge\tw(G)$ holds by virtue of $\o G\supseteq G$.}
be ``tree-like'' with a small treewidth $\tw(\o G)\ll n$. Concretely,
we construct the \emph{extended} aggregate sparsity graph $\o G=(V,\overline{E})$
by forcing each constraint matrix $A_{i}$ to be its own clique in
$G$:
\begin{gather}
V=\{1,2,\dots,n\},\quad\overline{E}=\spar(C)\cup\clique(A_{1})\cup\cdots\cup\clique(A_{m})\label{eq:Ebardef}\\
\text{where }\clique(A)=\{(i,j):A[i,k]\ne0\text{ or }A[k,j]\ne0\text{ for some }k\}.\nonumber 
\end{gather}
This is the union between $G$ and the constraint intersection graph~\citep{fulkerson1965incidence}
(or the dual graph~\citep{dong2021nearly} or the correlative sparsity~\citep{waki2006sums,kobayashi2008correlative})
for the rank-1 instance of (\ref{eq:sdp}). In other words, we add
a new edge $(i,j)$ to $E$ whenever $x_{i}$ and $x_{j}$ appear
together in a common constraint $x^{T}A_{k}x\le b_{k}$ for some $k$.
The fact that this contributes $\clique(A_{k})\subseteq\o E$ reflects
the reality that each $x^{T}A_{k}x\le b_{k}$ densely couples all
affected elements of $x$ together, forcing them to be optimized simultaneously.
In contrast, the cost $x^{T}Cx$ can be optimized sequentially over
the elements of $x$, which is why $\clique(C)$ is absent.

Our main result \thmref{total} says that if the extended graph $\o G$
has small treewidth $\tw(\o G)=O(1)$ with respect to $m$ and $n$,
then one can find a fill-reducing permutation $\Pi$ such that, after
reordering the data as $C\gets\Pi C\Pi^{T}$ and $A_{i}\gets\Pi A_{i}\Pi^{T}$,
the resulting instance of (\ref{eq:cc}) is solved by a general-purpose
interior-point method in guaranteed $O(m+n)$ time per-iteration,
over at most $O(\sqrt{m+n}\log(1/\epsilon))$ iterations. In practice,
a ``good enough'' permutation $\Pi$ is readily found by applying
an efficient fill-reducing heuristic to $\o G$, and a primal-dual
interior-point method often converges to $\epsilon$ accuracy in dimension-free
$O(\log(1/\epsilon))$ iterations (without the square-root factor).
If we take these two empirical observations\footnote{In \secref{exp}, we provide detailed numerical experiments to validate
these empirical observations on real-world datasets.} as formal assumptions, then a small treewidth $\tw(\o G)=O(1)$ in
the extended graph $\o G$ is indeed sufficient for chordal conversion
to solve the instance of (\ref{eq:sdp}) in $O((m+n)\log(1/\epsilon))$
empirical time.

In the case that $G$ and $\o G$ coincide, our analysis becomes exact;
a small treewidth in $\o G$ is both necessary and sufficient for
chordal conversion to achieve $O((m+n)\log(1/\epsilon))$ empirical
time. This is the case for the MAX-$k$-CUT relaxation~\citep{goemans1995improved,frieze1997improved}
and the Lovász theta problem~\citep{lovasz1979shannon}, two classic
SDPs that constitute a majority of test problems in the SDPLIB~\citep{borchers1999sdplib}
and the DIMACS~\citep{pataki2002dimacs} datasets. Here, $G=\o G$
because each constraint matrix $A_{i}$ indexes just a single matrix
element, as in $\inner{A_{i}}X=\alpha_{i}\cdot X[j_{i},k_{i}]$. Below,
we write $e_{j}$ as the $j$-th column of the identity matrix, and
$\one=[1,1,\dots,1]^{T}.$
\begin{example}[MAX-$k$-CUT]
\label{exa:maxcut}Let $C$ be the (weighted) Laplacian matrix for
a graph $\G=(\V,\E)$ with $\V=\{1,2,\dots,d\}$. \citet{frieze1997improved}
proposed a randomized algorithm to solve MAX $k$-CUT with an approximation
ratio of $1-1/k$ based on solving
\[
\max_{X\succeq0}\quad\frac{k-1}{2k}\inner CX\quad\text{s.t. }\quad\begin{array}{c}
X[i,i]=1\quad\text{for all }i\in\V\\
X[i,j]\ge\frac{-1}{k-1}\quad\text{for all }(i,j)\in\E
\end{array}
\]
The classic Goemans--Williamson 0.878~algorithm~\citep{goemans1995improved}
for MAXCUT is recovered by setting $k=2$ and removing the redundant
constraint $X[i,j]\ge-1$. \qed
\end{example}
\begin{example}[Lovász theta]
\label{exa:theta}The Lovász number $\vartheta(\G)$ of a graph $\G=(\V,\E)$
with $\V=\{1,2,\dots,d\}$ is the optimal value to the following~\citep{lovasz1979shannon}
\[
\min_{\lambda,y_{i,j}\in\R}\quad\lambda\quad\text{s.t. }\quad\one\one^{T}-\sum_{(i,j)\in\E}y_{i,j}(e_{i}e_{j}^{T}+e_{j}e_{i}^{T})\preceq\lambda I.
\]
Given that $\vartheta(\G)\ge1$ holds for all graphs $G$, dividing
through by $\lambda$ and applying the Schur complement lemma yields
a sparse reformulation
\[
\min_{X\succeq0}\quad\inner{\begin{bmatrix}I & \one\\
\one^{T} & 0
\end{bmatrix}}X\quad\text{s.t. }\quad\begin{array}{c}
X[d+1,d+1]=1,\\
X[i,j]=0\quad\text{for all }(i,j)\in\E.
\end{array}
\]
\qed
\end{example}
We also have $G=\o G$ in the sensor network localization problem,
one of the first successful applications of chordal conversion to
a real-world problem~\citep{kim2009exploiting}, because $\spar(A_{i})=\clique(A_{i})$
holds for all $i$. (Assume without loss of generality that each $a_{k}$
below contains only nonzero elements.)
\begin{example}[Sensor network localization]
\label{exa:sensor}We seek to find unknown sensor points $x_{1},\dots,x_{n}\in\R^{d}$
such that 
\[
\|x_{i}-x_{j}\|=r_{i,j}\quad\text{for all }(i,j)\in N_{x},\qquad\|x_{i}-a_{k}\|=\rho_{i,k}\quad\text{for all }(i,k)\in N_{a}
\]
given a subset $N_{x}$ of distances $r_{i,j}$ between the $i$-th
and $j$-th sensors, and a subset $N_{a}$ of distances $\rho_{i,k}$
between the $i$-th sensor and the $k$-th known anchor point $a_{k}\in\R^{d}$.
Biswas and Ye~\citep{biswas2004semidefinite} proposed the following
SDP relaxation
\[
\min_{X\succeq0}\quad\inner 0X\quad\text{s.t. }\begin{array}{c}
\inner{\begin{bmatrix}1 & -1\\
-1 & 1
\end{bmatrix}}{\begin{bmatrix}X[i,i] & X[i,j]\\
X[i,j] & X[j,j]
\end{bmatrix}}=r_{i,j}^{2}\quad\text{for all }(i,j)\in N_{x},\\
\inner{\begin{bmatrix}1 & -a_{k}^{T}\\
-a_{k} & a_{k}a_{k}^{T}
\end{bmatrix}}{\begin{bmatrix}X[i,i] & X[i,K]\\
X[K,i] & X[K,K]
\end{bmatrix}}=\rho_{i,k}^{2}\quad\text{for all }(i,k)\in N_{a},\\
X[K,K]=I_{d},
\end{array}
\]
where $K=(n+1,\dots,n+d)$.\qed
\end{example}

Our result can also be applied to the chordal conversion of SDPs that
arise in polynomial optimization. The following class of polynomial
optimization covers many of the unconstrained test problems in the
original paper~\citep{waki2006sums} that first introduced chordal
conversion to this setting. Below, a matrix $U\in\R^{p\times p}$
is said to be Hankel if its skew-diagonals are constant, i.e. $U[i,j]=U[i+1,j-1]$
for all $1\le i,j\le p$. We see that $G=\o G$ holds because the
Hankel constraint is dense over its support. 
\begin{example}[Unconstrained polynomial optimization]
Given $C_{i,j}\in\R^{p\times p}$ for $i,j\in\{1,2,\dots,n\}$, consider
the following
\[
\min_{x_{1},\dots,x_{n}}\sum_{j=1}^{n}u_{i}^{T}C_{i,j}u_{j}\text{ where }u_{j}=[1,x_{j},x_{j}^{2},\dots,x_{j}^{p-1}]^{T}.
\]
The basic Lasserre--Parrilo SDP relaxation~\citep{parrilo2000structured,lasserre2001global}
for this problem (without cross terms) reads
\[
\min_{[U_{i,j}]_{i,j=1}^{n}\succeq0}\quad\sum_{j=1}^{n}\inner{C_{i,j}}{U_{i,j}}\quad\text{s.t. }\quad U_{i,i}\text{ is Hankel},\quad U_{i,i}[1,1]=1\quad\text{for all }i,
\]
where $[U_{i,j}]_{i,j=1}^{n}$ denotes an $np\times np$ matrix, comprising
of $n\times n$ blocks of $p\times p$, with $U_{i,j}\in\R^{p\times p}$
as its $(i,j)$-th block. \qed
\end{example}
But the real strength of our result is its ability to handle cases
for which $G\subset\o G$ holds strictly. An important real-world
example is the SDP relaxation of the AC optimal power flow problem~\citep{bai2008semidefinite,lavaei2012zero},
which plays a crucial role in the planning and operations of electricity
grids. 

\begin{example}[AC optimal power flow relaxation]
\label{exa:acopf}Given a graph $\G=(\V,\E)$ on $d$ vertices $\V=\{1,\dots,d\}$,
we say that ${A_{i}}\in\S^{2d}$ implements a\emph{ }power flow constraint
at vertex $k\in\V$ if it can be written in terms of $\alpha_{i,k}\in\S^{2}$
and $\alpha_{i,j}\in\R^{2\times2}$ as:
\[
{A_{i}}=e_{k}e_{k}^{T}\otimes{\alpha_{i,k}}+\frac{1}{2}\sum_{(j,k)\in E}\left[e_{j}e_{k}^{T}\otimes{\alpha_{i,j}}+e_{k}e_{j}^{T}\otimes{\alpha_{i,j}^{T}}\right].
\]
An instance of the AC optimal power flow relaxation is written
\[
\min_{X\succeq0}\quad\sum_{j}\inner{C_{j}}X\quad\text{ s.t. }\quad b_{i}^{\lb}\le\inner{A_{i}}X\le b_{i}^{\ub}
\]
in which every $A_{i}$ and $C_{j}$ implements a power flow constraint
at some vertex $v\in\V$. \qed
\end{example}
It can be verified that $\tw(G)=2\cdot\tw(\G)$ and $\tw(\o G)=2\cdot\tw(\mathcal{G}^{2})$,
where the square graph $\mathcal{G}^{2}$ is defined so that $(i,j)\in\G^{2}$
if and only if $i$ and $j$ are at most a distance of $2$ away in
$\G$.  In fact, knowing that an electric grid $\G$ is ``tree-like''
does not in itself guarantee chordal conversion to be fast, because
it does not imply that $\G^{2}$ would also be ``tree-like''. While
chordal conversion is already widely used to solve the AC optimal
power flow relaxation~\citep{jabr2012exploiting,madani2016promises,Eltved2019,molzahn2019survey},
our finding in \secref{exp} that $\tw(\G^{2})\ll n$ holds for real-world
power systems (see \tabref{treedecomp}) provides the first definitive
explanation for its observed $O((m+n)\log(1/\epsilon))$ empirical
time complexity.

It remains future work to understand the cases where $\tw(G)$ and
$\tw(\o G)$ are very different. In the case of the AC optimal power
flow relaxation, it is not difficult to construct a counterexample
where $\tw(G)=2$ and $\tw(\o G)=n-2$ (set $\G$ to be the star graph,
so that $\G^{2}$ is the complete graph) and observe $\Omega(n^{3})$
per-iteration cost after chordal conversion. This counterexample,
along with the trivial \exaref{diag}, both hint at the possibility
that a small treewidth in $\o G$ is both necessary and sufficient
for $O(m+n)$ time per-iteration, but more work is needed establish
this rigorously.

\subsection{Prior results: Complexity of clique-tree conversion}

Our result is related to a prior work of \citet{zhang2020sparse}
that studied a different conversion method called\emph{ clique-tree
conversion}, also due to \citet{fukuda2001exploiting}. This can best
be understood as a second step of conversion added on top of chordal
conversion. Recall that chordal conversion converts (\ref{eq:sdp})
into (\ref{eq:cc}), and then solves the latter using an interior-point
method. Clique-tree conversion further converts (\ref{eq:cc}) into
the following problem by splitting each submatrix $X_{j}\equiv X[J_{j},J_{j}]$
into its own variable:
\begin{equation}
\min_{X_{1},\dots,X_{n}}\;\sum_{j=1}^{n}\inner{C_{j}}{X_{j}}\text{ s.t. }\begin{array}{c}
\sum_{j=1}^{n}\inner{A_{i,j}}{X_{j}}\le b_{i}\quad\text{for all }i\in\{1,2,\dots,m\},\\
X_{j}\succeq0\quad\text{for all }j\in\{1,2,\dots,n\},\\
\mathcal{N}_{u,v}(X_{v})=\mathcal{N}_{v,u}(X_{u})\quad\text{for all }(u,v)\in\mathcal{T}.
\end{array}\tag{CTC}\label{eq:ctc}
\end{equation}
 The constraint $\mathcal{N}_{u,v}(X_{v})=\mathcal{N}_{v,u}(X_{u})$
is added to enforce agreement between overlapping submatrices, over
the edges of the eponymous clique tree $\mathcal{T}$. 

The point of converting (\ref{eq:cc}) to (\ref{eq:ctc}) is to force
a favorable sparsity pattern in the Schur complement equations solved
at each iteration of the interior-point method, which is known as
the \emph{Schur complement sparsity}~\citep[Section 13.1]{vandenberghe2015chordal}
or the \emph{correlative sparsity}~\citep{kobayashi2008correlative,kim2011exploiting}.
In fact, \citet{zhang2020sparse} pointed out that the Schur complement
sparsity of (\ref{eq:ctc}) is particularly simple to analyze. Under
small treewidth assumptions, they proved that an interior-point method
solves (\ref{eq:ctc}) in guaranteed $O(m+n)$ time per-iteration,
over at most $O(\sqrt{m+n}\log(1/\epsilon))$ iterations; see also
\citet{gu2022faster}. But a major weakness of this result is that
it critically hinges on the second step of conversion, from (\ref{eq:cc})
to (\ref{eq:ctc}). On a basic level, it does not explain the plethora
of numerical experiments showing that interior-point methods are able
to solve (\ref{eq:cc}) directly in $O(m+n)$ time per-iteration \emph{without}
a second conversion step~\citep{waki2006sums,kim2009exploiting,kim2011exploiting,jabr2012exploiting}. 

Indeed, the numerical experiments of \citet{kim2011exploiting} strongly
suggest, for instances of (\ref{eq:ctc}) with favorable Schur complement
sparsity, that the Schur complement sparsity of (\ref{eq:cc}) had
already been favorable in the first place, so the second conversion
step was unnecessary, other than for the sake of a proof. Unfortunately,
the Schur complement sparsity of (\ref{eq:cc}) is much more complicated
than that of (\ref{eq:ctc}). Prior to this work, \citet*{kobayashi2008correlative}
provided a characterization for when the Schur complement sparsity
of (\ref{eq:cc}) is favorable. However, their characterization can
only be checked numerically, in a similar amount of work as performing
a single iteration of the interior-point method, and so gives no deeper
insight on what classes of SDPs can be efficiently solved. Our main
technical contribution in this paper is \thmref{front}, which characterizes
the complicated Schur complement sparsity in terms of the much-simpler
extended sparsity $\o E$. It is this simplicity that allowed us to
analyze many successful applications of chordal conversion, as detailed
in the previous section. 

In practice, the second conversion step from (\ref{eq:cc}) to (\ref{eq:ctc})
results in a massive performance penalty, both in preprocessing time
and in the solution time. In our large-scale experiments in \secref{exp},
the second step of converting from (\ref{eq:cc}) into (\ref{eq:ctc})
can sometimes take more than 100 times longer than the first step
of converting (\ref{eq:sdp}) into (\ref{eq:cc}). Also, we find that
the state-of-the-art solver MOSEK~\citep{mosek2019} takes a factor
of 2 to 100 times more time to solve (\ref{eq:ctc}) than the original
instance of (\ref{eq:cc}). Previously, clique-tree conversion was
used to solve an instance of AC optimal power flow relaxation with
$n=8.2\times10^{4}$ and $m\approx2.5\times10^{5}$ on a high-performance
computing (HPC) node with 24 cores and 240 GB memory in 8 hours~\citep{Eltved2019}.
In this paper, we solve this same problem using chordal conversion
on a modest workstation with 4 cores and 32 GB of RAM in just 4 hours. 

\subsection{Other approaches}

In this paper, we focus on chordal conversion in the context of high-accuracy
interior-point methods. We mention that chordal conversion has also
been used to reduce the per-iteration cost of first-order methods
to $O(m+n)$ time~\citep{sun2014decomposition,zheng2020chordal},
but these can require many iterations to converge to high accuracy.
Also, nonconvex approaches~\citep{burer2003nonlinear,boumal2020deterministic}
have recently become popular, but it remains unclear how these could
be made to benefit from chordal conversion. 

The recent preprint of \citet{gu2022faster} combined the fast interior-point
method of \citet*[Theorem 1.3]{dong2021nearly} and the clique-tree
conversion formulation of \citet{zhang2020sparse} to prove that,
if the extended graph $\o G$ has small treewidth, then there exists
an algorithm to solve (\ref{eq:sdp}) to $\epsilon$ accuracy in $\tilde{O}((m+n)\log(1/\epsilon))$
worst-case time. This improves over our $O(m+n)$ time per-iteration
figure, which must be spread across $O(\sqrt{m+n}\log(1/\epsilon))$
worst-case iterations, for a total of $O((m+n)^{1.5}\log(1/\epsilon))$
worst-case time. However, it is important to point out that these
``fast'' interior-point methods~\citep{dong2021nearly,gu2022faster}
are purely theoretical; their analysis hides numerous leading constants,
and it is unclear whether a real-world implementation could be made
competitive against state-of-the-art solvers. On the other hand, primal-dual
solvers like MOSEK~\citep{mosek2019} typically converge to $\epsilon$
accuracy in dimension-free $O(\log(1/\epsilon))$ iterations (see
our experiments in \secref{exp}), so in practice, our algorithm is
already able to solve (\ref{eq:sdp}) to $\epsilon$ accuracy in $O((m+n)\log(1/\epsilon))$
empirical time.

\section{Preliminaries}

\subsection{Notations and basic definitions}

Write $\R^{m\times n}$ as the set of $m\times n$ matrices with real
coefficients, with associated matrix inner product $\inner AB=\tr A^{T}B$
and norm $\|A\|_{F}=\sqrt{\inner AA}$. Write $\S^{n}\subseteq\R^{n\times n}$
as the set of $n\times n$ real symmetric matrices, meaning that $X=X^{T}$
holds for all $X\in\S^{n}$, and write $\S_{+}^{n}\subseteq\S^{n}$
as the set of symmetric positive semidefinite matrices. Write $\R_{+}^{n}\subseteq\R^{n}$
as the usual positive orthant.

The set of $n\times n$ symmetric matrices with sparsity pattern $E$
can be defined as
\[
\S_{E}^{n}\equiv\left\{ X\in\S^{n}:X[i,j]=X[j,i]=0\quad\text{for all }i\ne j,\quad(i,j)\notin E\right\} .
\]
Conversely, the minimum sparsity pattern of a \emph{symmetric} matrix
$X\in\S^{n}$ is denoted
\[
\spar(X)\equiv\{(i,j):X[i,j]\ne0,\quad i\ge j\}.
\]
We also write $\spar(M)\equiv\spar(M+M^{T})$ for a nonsymmetric matrix
$M$ where there is no confusion. We write $\proj_{E}(M)\equiv\arg\min_{X\in\S_{E}^{n}}\|M-X\|_{F}$
as the projection of $M\in\S^{n}$ onto the sparsity pattern $E$. 

We define the dense sparsity pattern induced by $J\subseteq\{1,2,\dots,n\}$
as follows
\[
\clique(J)=\{(i,j):i,j\in J,\quad i\ge j\}.
\]
We also define the vertex support of a possibly nonsymmetric matrix
$M$ as the following 
\[
\supp(M)=\{i:M[i,j]\ne0\quad\text{for some }j\}.
\]
We write $\clique(A)\equiv\clique(\supp(A))$ where there is no confusion.
This notation is motivated by the fact that $\spar(A)\subseteq\clique(A)$
for $A\in\S^{n}$, and $\spar(PDP^{T})\subseteq\clique(P)$ for $P\in\R^{n\times d}$
and dense $D\in\S^{d}$. 

Let $F$ be a sparsity pattern of order $n$ that \emph{contains all
of its diagonal elements}, as in $F\supseteq\{(i,i):i\in\{1,2,\dots,n\}\}$.
In this case, $\dim(\S_{F}^{n})=|F|$ holds exactly, so we define
a \emph{symmetric vectorization} operator $\svec_{F}:\S_{F}^{n}\to\R^{|F|}$
to implement an isometry with the usual Euclidean space, as in 
\[
\inner{\svec_{F}(X)}{\svec_{F}(Y)}=\inner XY\text{ for all }X,Y\in{\S_{F}^{n}}.
\]
We will explicitly require $\svec_{F}(\cdot)$ to be defined according
to a column-stacking construction
\[
\svec_{F}(X)=(x_{j})_{j=1}^{n}\quad\text{where }x_{j}=\left(X[i,i],\quad\sqrt{2}\cdot(X[i,j]:i>j,\quad(i,j)\in F)\right).
\]
We also define a companion indexing operator $\idx_{F}(\cdot,\cdot)$
to index elements of the vectorization $x=\svec_{F}(X)$:
\begin{gather*}
x[\idx_{F}(i,i)]=X[i,i]\quad\text{for all }i\in\{1,2,\dots,n\},\\
x[\idx_{F}(i,j)]=\sqrt{2}X[i,j]\quad\text{for all }i>j,\quad(i,j)\in F.
\end{gather*}
As we will see later, the correctness of our proof crucially relies
on the fact that $\idx_{F}(\cdot,\cdot)$ implements a \emph{raster
ordering} over the elements of $F$. 

\subsection{Sparse Cholesky factorization}

To solve $Sx=b$ with $S\succ0$ via Cholesky factorization, we first
compute the lower-triangular Cholesky factor $L=\chol(S)$ according
to the following recursive rule 
\[
\chol\left(\begin{bmatrix}\alpha & b^{T}\\
b & D
\end{bmatrix}\right)=\begin{bmatrix}\sqrt{\alpha} & 0\\
\frac{1}{\sqrt{\alpha}}b & \chol\left(D-\frac{1}{\alpha}bb^{T}\right)
\end{bmatrix},\quad\chol(\alpha)=\sqrt{\alpha},
\]
and then solve two triangular systems $Ly=b$ and $L^{T}x=y$ via
back-substitution. If $S$ is sparse, then $L=\chol(S)$ may also
be sparse. The sparsity pattern of $L$ can be directly computed from
the sparsity pattern of $S$, and without needing to examine the numerical
values of its nonzeros.
\begin{definition}[Symbolic Cholesky]
\label{def:symbchol}The \emph{symbolic Cholesky factor} $\chol(E)$
of a sparsity pattern $E$ of order $n$ is defined as $\chol(E)\equiv E_{n+1}$
where $E_{1}=E$ and
\[
E_{k+1}=E_{k}\cup(k,k)\cup\{(i,j):(i,k)\in E_{k},\quad(j,k)\in E_{k}\text{ for }i>j>k\}.
\]
\end{definition}
One can verify that $\chol(\spar(S))=\spar(\chol(S))$. Note that
$\chol(E)$ can be computed from $E$ in $O(|\chol(E)|)$ time and
memory~\citep[Theorem 5.4.4]{george1981computer}. The efficiency
of a sparsity-exploiting algorithm for factorizing $L=\chol(S)$ and
solving $Ly=b$ and $L^{T}x=y$ is determined by the \emph{frontsize}
of the sparse matrix $S$. 
\begin{definition}[Frontsize]
\label{def:frontsize}The \emph{frontsize} $\omega(E)$ of a sparsity
pattern $E$ is defined $\max_{j}|\col_{F}(j)|$ where $F=\chol(E)$
and $\col_{F}(j)\equiv\{j\}\cup\{i>j:(i,j)\in F\}$. The frontsize
$\omega(S)\equiv\omega(\spar(S))$ of a symmetric matrix $S$ is the
frontsize of its minimum sparsity pattern. 
\end{definition}
Intuitively, the frontsize $\omega(S)$ is the maximum number of nonzero
elements in a single column of the Cholesky factor $L=\chol(S)$.
The following is well-known~\citep{george1981computer}.
\begin{proposition}[Sparse Cholesky factorization]
\label{prop:chol}Given $S\in\S^{n},$ $S\succ0$, let $\omega\equiv\omega(S)$.
Sparse Cholesky factorization factors $L=\chol(S)$ in $\Tim$ arithmetic
operations and $M$ units of memory, where 
\begin{align*}
\frac{1}{6}(\omega-1)^{3}+n & \le\Tim\le\omega^{2}\cdot n, & \frac{1}{2}(\omega-1)^{2}+n & \le\Mem\le\omega\cdot n.
\end{align*}
\end{proposition}
\begin{proof}
Let $\omega_{j}\equiv|\col_{F}(j)|$. By inspection, $\Tim=\sum_{j=1}^{n}\frac{1}{2}\omega_{j}(\omega_{j}+1)$
and $\Mem=\sum_{j=1}^{n}\omega_{j}$. The bounds follow by substituting
$\omega\ge\omega_{j+1}\ge\omega_{j}-1$. Indeed,
\[
\Mem=\sum_{j=1}^{n}\omega_{j}\ge\underbrace{\omega+(\omega-1)+\cdots+1}_{\omega\text{ terms}}+\underbrace{1+\cdots+1}_{n-\omega\text{ terms}}=\frac{1}{2}\omega(\omega+1)+(n-\omega)
\]
and similarly $\Tim=\sum_{j=1}^{n}\frac{1}{2}\omega_{j}(\omega_{j}+1)\ge\frac{1}{6}\omega(\omega+1)(\omega+2)+(n-\omega).$
\qed
\end{proof}

Note that \propref{chol} is sharp up to small additive constants:
the upper-bound is essentially attained by banded matrices of bandwidth
$\omega$, while the lower-bound is essentially attained by a matrix
that contains a single dense block of size $\omega$.

\subsection{Minimum frontsize and treewidth}

\global\long\def\Perm{\mathrm{Perm}}%
The cost of solving $Sx=b$ with sparse $S\succ0$ can usually be
reduced by first permuting the rows and columns of the matrix symmetrically,
and then solving $(\Pi S\Pi^{T})\Pi x=\Pi y$ for some permutation
matrix $\Pi$. For $E=\spar(S)$, we write $E_{\Pi}\equiv\spar(\Pi S\Pi^{T})$
to denote its permuted sparsity pattern. It is a fundamental result
in graph theory and linear algebra that the problem of minimizing
the frontsize $\omega(E_{\Pi})$ over the set of $n\times n$ permutation
matrices $\Pi\in\Perm(n)$ is the same problem as computing the treewidth
of the graph $G=(V,E)$.

\begin{definition}[Treewidth]
\label{def:tw}A tree decomposition of a graph $G=(V,E)$ is a pair
$(\{J_{j}\},T)$ in which each \emph{bag} $J_{j}\subseteq V$ is a
subset of vertices and $T$ is a tree such that:
\begin{itemize}
\item (Vertex cover) $\bigcup_{j}J_{j}=V$;
\item (Edge cover) $\bigcup_{j}(J_{j}\times J_{j})\supseteq E$;
\item (Running intersection) $J_{i}\cap J_{j}\subseteq J_{k}$ for every
$k$ on the path of $i$ to $j$ on $T$.
\end{itemize}
The \emph{width} of the tree decomposition is $\max_{j}|J_{j}|-1$.
The \emph{treewidth} of $G$, denoted $\tw(G)$, is the minimum width
over all valid tree decompositions on $G$.
\end{definition}
The connection is an immediate corollary of the following result,
which establishes an equivalence between tree decompositions and the
sparsity pattern of Cholesky factors. 
\begin{proposition}[Perfect elimination ordering]
Given a sparsity pattern $E$ of order $n$, let $G=(V,E)$ where
$V=\{1,2,\dots,n\}$. For every tree decomposition of $G$ with width
$\tau$, there exists a perfect elimination ordering $\Pi\in\Perm(n)$
such that $\omega(E_{\Pi})=1+\tau$. 
\end{proposition}
We defer a proof to the texts~\citep{george1981computer,vandenberghe2015chordal},
and only note that, given a tree decomposition of width $\tau$, the
corresponding perfect elimination ordering $\Pi$ can be found in
$O((1+\tau)\cdot n)$ time. 
\begin{corollary}
\label{cor:treewidth}We have $1+\tw(G)=\min_{\Pi\in\mathrm{Perm}(n)}\omega(E_{\Pi})$.
\end{corollary}
As a purely theoretical result, if we assume that $\tw(G)=O(1)$ with
respect to the number of vertices $n$, then a choice of $\Pi\in\Perm(n)$
that sets $\omega(E_{\Pi})=O(1)$ can be found in $O(n)$ time~\citep{fomin2018fully}
(and so the problem is no longer NP-hard). In practice, it is much
faster to use simple greedy heuristics~\citep{amestoy2004algorithm},
which often find ``good enough'' choices of $\Pi$ that yield very
small values of $\omega(E_{\Pi})$, without a rigorous guarantee of
quality. 

\subsection{General-purpose interior-point methods}

The basic approach for solving an SDP using a general-purpose solver
is to reformulate the problem into the primal or the dual of the standard-form
linear conic program 
\[
\min_{x\in{\R^{q}}}\{\c^{T}x:\A x={\b},x\in\K\}\ge\max_{y\in{\R^{p}}}\{\b^{T}y:\c-\A^{T}y\in{\K_{*}}\}
\]
 where the data are the matrix $\A\in\R^{p\times q}$, vectors $\b\in\R^{p}$
and $\c\in\R^{q}$, and the problem closed convex cone $\K\subseteq\R^{q}$,
and the notation $\K_{*}$ means the dual cone of $\K$. We specify
the following basic assumptions on this problem to ensure that it
can be solved in polynomial time using a self-dual embedding~\citep{de2000self}.
Below, we denote $\one_{\K}$ as the identity element on the cone
$\K$, and recall that every semidefinite cone is self-dual $\K=\K_{*}$. 
\begin{definition}[Standard-form SDP]
\label{def:clp}We say that the problem data $\A\in\R^{p\times q},$
$\b\in\R^{p},$ $\c\in\R^{q},$ and $\K\subseteq\R^{q}$ describe
an SDP in $(n,\omega)$-\emph{standard form} if:
\begin{enumerate}
\item (Dimensions) The cone $\K=\svec(\S_{+}^{\omega_{1}})\times\cdots\times\svec(\S_{+}^{\omega_{\ell}})$
is the Cartesian product of semidefinite cones whose orders $\omega_{1},\omega_{2},\dots,\omega_{\ell}$
satisfy 
\[
\omega=\max_{i}\omega_{i},\qquad n=\sum_{i=1}^{\ell}\omega_{i},\qquad{q}=\frac{1}{2}\sum_{i=1}^{\ell}\omega_{i}(\omega_{i}+1).
\]
\item (Linear independence) $\A^{T}y=0$ holds if and only if $y=0$. 
\item (Strong duality is attained) There exist a choice of $x^{\star},y^{\star}$
that satisfy
\begin{gather*}
\A x^{\star}=\b,\quad x^{\star}\in\K,\quad\c-\A^{T}y^{\star}\in\K,\quad\c^{T}x^{\star}=\b^{T}y^{\star}.
\end{gather*}
\end{enumerate}
\end{definition}
\begin{definition}[General-purpose solver]
\label{def:ipm}We say that $\IPM$ implements a \emph{general-purpose
solver} if it satisfies the following conditions
\begin{enumerate}
\item (Iteration count) Given data $(\A,\b,\c,\K)$ in $(n,\omega)$-standard-form,
calling $(x,y)=\IPM(\epsilon,\A,\b,\c,\K)$ yields iterates $(x,y)\in\K\times{\R^{p}}$
that satisfy the following, in at most $O(\sqrt{n}\log(1/\epsilon))$
iterations
\[
\|\A x-\b\|\le\epsilon,\quad{\c-\A^{T}y+\epsilon\cdot\one_{\K}\in\K,}\quad\c^{T}x-\b^{T}y\le\epsilon\cdot n.
\]
\item (Per-iteration costs) Each iteration costs an overhead of $O(\omega^{2}n)$
time and $O(\omega n)$ memory, plus the cost of solving $O(1)$ instances
of the \emph{Schur complement equation} 
\[
\A\nabla^{2}f(w)\A^{T}\Delta y=r,\qquad f(w)=-\log\det(w)
\]
by forming $\Hess=\A\nabla^{2}f(w)\A^{T}$, factoring $\L=\chol(\Pi\Hess\Pi^{T})$
and then solving $\L z={\Pi}r$ and $\L^{T}(\Pi\Delta y)=z$. Here,
the fill-reducing permutation $\Pi$ is required to be no worse than
the natural ordering, as in $\omega(\Pi\Hess\Pi^{T})\le\omega(\Hess)$.
\end{enumerate}
\end{definition}
Note that \defref{ipm} is rigorously satisfied by SeDuMi~\citep{sturm1999using,sturm2002implementation},
SDPT3~\citep{toh1999sdpt3,tutuncu2003solving}, and MOSEK~\citep{andersen2000mosek,andersen2003implementing,mosek2019}.
Given that the correctness of our overall claims crucially depends
on the characterization in \defref{ipm}, we state a concrete interior-point
method in \appref{ipm} that implements these specifications. 

\section{Main results}

\begin{algorithm}[t]
\caption{\label{alg:myalg}Chordal conversion}

\textbf{Input.} Accuracy parameter $\epsilon>0$, problem data $C,A_{1},A_{2},\dots,A_{m}\in\S^{n}$,
$b\in\R^{m}$, fill-reducing permutation $\Pi$.

\textbf{Output.} Approximate solutions $U\in\R^{n\times\omega}$ and
$v\in\R^{m}$ to the following primal-dual pair:
\[
\min_{U}\left\{ \inner C{UU^{T}}:\inner{A_{i}}{UU^{T}}\le b_{i}\text{ for all }i\right\} \ge\max_{v\le0}\left\{ \inner bv:{\textstyle \sum_{i}}v_{i}A_{i}\preceq C\right\} .
\]

\textbf{Algorithm.}
\begin{enumerate}
\item (Symbolic factorization) Pre-order all data matrices $\tilde{A}_{i}=\Pi A_{i}\Pi^{T}$
and $\tilde{C}=\Pi C\Pi^{T}$. Compute the permuted aggregate sparsity
pattern $E=\spar(\tilde{C})\cup\bigcup_{i}\spar(\tilde{A}_{i})$,
its lower-triangular symbolic Cholesky factor $F=\chol(E)$, and define
the following
\[
J_{j}=\col_{F}(j)\equiv\{j\}\cup\{i>j:(i,j)\in F\},\quad\omega_{j}\equiv|J_{j}|,\quad\omega\equiv\max_{j}\omega_{j}.
\]
\item (Numerical solution) Call $(x,y)=\IPM(\epsilon,\A,\b,\c,\K)$ where
$\IPM$ is a general-purpose solver (\defref{ipm}), and the problem
data $\A,\b,\c,\K$ implement the following
\[
Y^{\star}=\arg\max_{Y\in\S_{F}^{n}}\inner{-\tilde{C}}Y\text{ s.t. }\begin{array}{cc}
b_{i}-\inner{\tilde{A}_{i}}Y\ge0 & \quad\text{for all }i\in\{1,2,\dots,m\}\\
Y[J_{j},J_{j}]\succeq0 & \quad\text{for all }j\in\{1,2,\dots,n\}
\end{array}
\]
as an instance of $y^{\star}=\arg\max_{y}\{\b^{T}y:\c-\A^{T}y\in\K\}$
with $y^{\star}=\svec_{F}(Y^{\star})$.
\item (Back substitution) Recover $Y\in\S_{F}^{n}$ from $y=\svec_{F}(Y)$,
and compute $\delta=-\min_{j}\{0,\lambda_{\min}(Y[J_{j},J_{j}])\}$.
Solve the positive semidefinite matrix completion
\[
\text{find }\tilde{U}\in\R^{n\times\omega}\text{ such that }(\tilde{U}\tilde{U}^{T})[J_{j},J_{j}]=Y[J_{j},J_{j}]+\delta I\text{ for all }j
\]
using \citep[Algorithm~2]{sun2015decomposition}. Output $v=-(x_{i})_{i=1}^{m}$
and $U=\Pi^{T}\tilde{U}$.
\end{enumerate}
\end{algorithm}
 \algref{myalg} summarizes the standard implementation of chordal
conversion, which is known as the ``d-space conversion method using
basis representation'' in \citet{kim2011exploiting}. Our only modification
is to recover $X=UU^{T}$ from $Y=\proj_{F}(X)$ in Step 3 using the
\emph{low-rank} chordal completion~\citep[Theorem~1.5]{dancis1992positive},
instead of the \emph{maximum determinant} chordal completion~\citep[Theorem~2]{grone1984positive}
as originally proposed by \citet{fukuda2001exploiting}. We note that
both recovery procedures have the same complexity of $O(\omega^{3}n)$
time and $O(\omega^{2}n)$ memory, but the former puts $X$ in a more
convenient form~\citep{sun2015decomposition}.

The cost of \algref{myalg} is dominated by the cost of solving the
\emph{Schur complement equation} at each iteration of the interior-point
method in Step 2. At the heart of this paper is a simple but precise
upper-bound on the frontsize of its sparsity pattern $E^{(2)}$, given
in terms of the extended sparsity pattern $\o E$.  Concretely, our
result says that if $\omega(\o E)=O(1)$, then $\omega(E^{(2)})=O(1)$,
so the Schur complement matrix can be formed, factored, and backsubstituted
in $O(m+n)$ time. Hence, the per-iteration cost of the interior-point
method is also $O(m+n)$ time.

To state the Schur complement sparsity $E^{(2)}$ explicitly, note
that the problem data $(\A,\b,\c,\K)$ in Step 2 of \algref{myalg}
are \begin{subequations}\label{eq:AKdef}
\begin{equation}
\begin{array}{c}
\A=[\svec_{F}(A_{1}),\dots,\svec_{F}(A_{m}),-\P_{1},\dots,-\P_{n}],\\
\b=-\svec_{F}(C),\qquad\c=(b,0),\\
\K=\R_{+}^{m}\times\svec(\S_{+}^{\omega_{1}})\times\svec(\S_{+}^{\omega_{2}})\times\cdots\times\svec(\S_{+}^{\omega_{n}})
\end{array}\label{eq:AKdef-1}
\end{equation}
where $\omega_{j}\equiv|\col_{F}(j)|$ and each $\P_{j}$ is implicitly
defined to satisfy 
\begin{equation}
\P_{j}^{T}\svec_{F}(Y)=\svec(Y[\col_{F}(j),\col_{F}(j)])\quad\text{for all }Y\in\S_{F}^{n}.\label{eq:AKdef-2}
\end{equation}
\end{subequations}The resulting Schur complement matrix reads
\begin{gather}
\A\nabla^{2}f(w)\A^{T}=\sum_{i=1}^{m}d_{i}\svec_{F}(A_{i})\svec_{F}(A_{i})^{T}+\sum_{j=1}^{n}\P_{j}\D_{j}\P_{j}^{T}\label{eq:Hessdef}\\
\text{where }d_{i}=w_{i}^{-2},\qquad\D_{j}\svec(X_{j})=\svec(W_{j}^{-1}X_{j}W_{j}^{-1})\nonumber 
\end{gather}
and $w=(w_{1},\dots,w_{m},\svec(W_{1}),\dots,\svec(W_{n}))\in\interior(\K)$
is a scaling point. The associated sparsity pattern, aggregated over
all possible choices of scaling $w$, is as follows
\begin{align}
E^{(2)} & =\{(i,j):(\A\nabla^{2}f(w)\A^{T})[i,j]\ne0\text{ for some }w\in\interior(\K)\}\nonumber \\
 & =\left(\bigcup_{i=1}^{m}\spar(a_{i}a_{i}^{T})\right)\cup\left(\bigcup_{j=1}^{n}\spar(\P_{j}\D_{j}\P_{j}^{T})\right)\nonumber \\
 & =\left(\bigcup_{i=1}^{m}\clique(\supp(a_{i}))\right)\cup\left(\bigcup_{j=1}^{n}\clique(\supp(\P_{j}))\right)\label{eq:E2def}
\end{align}
where we have written $a_{i}=\svec_{F}(A_{i})$. The result below
says that if $\omega(\o E)=O(1)$, then $\omega(\o E^{(2)})=O(1)$. 

\begin{theorem}[Frontsize of Schur complement sparsity]
\label{thm:front}Given $C,A_{1},A_{2},\dots,A_{m}\in\S^{n}$, define
$\A,\K$ as in (\ref{eq:AKdef}), and define $E^{(2)}$ as in (\ref{eq:E2def}).
We have 
\[
\frac{1}{2}\omega(\omega+1)\le\omega(E^{(2)})\le\frac{1}{2}\o{\omega}(\o{\omega}+1)
\]
where $\omega\equiv\omega(E)$ and $\o{\omega}=\omega(\o E)$ and
$E,\o E$ are defined in (\ref{eq:Edef}) and (\ref{eq:Ebardef}).
Moreover, if $E=\o E$, then we also have $\chol(E^{(2)})=E^{(2)}$. 
\end{theorem}
In cases where $E=\o E$, as in the MAX-$k$-CUT relaxation (\exaref{maxcut})
and the Lovász theta problem (\exaref{theta}), \thmref{front} predicts
that the Schur complement matrix $\Hess=\A\nabla^{2}f(w)\A^{T}$ can
be factored $\L=\chol(\Hess)$ with \emph{zero fill-in}, meaning that
$\spar(\L+\L^{T})=\spar(\Hess)$. More generally, if $\omega<\o{\omega}$
holds with a small gap, then we would also expect $\Hess$ to factor
with very little fill-in.

As previously pointed out by \citet{kobayashi2008correlative}, if
the Schur complement sparsity $E^{(2)}$ is known to have frontsize
$\omega(E^{(2)})=O(1)$, then the Schur complement matrix $\Hess=\A\nabla^{2}f(w)\A$
can be formed, factored, and backsubstituted in $O(m+n)$ time. Hence,
the per-iteration cost of the interior-point method is also $O(m+n)$
time. 
\begin{corollary}[Cost of Schur complement equation]
\label{cor:normal}Given the data matrix $\A$, scaling point $w\in\interior(\K)$,
and right-hand side $g$, define $\Hess=\A\nabla^{2}f(w)\A^{T}$ as
in (\ref{eq:E2def}). Suppose that all columns of $\A$ are nonzero,
and all scaling matrices $\D_{j}$ are fully dense. Then, it takes
$\Tim$ arithmetic operations and $\Mem$ units of storage to form
$\Hess$ and solve $\Hess\Delta y=g$, where
\begin{gather*}
\frac{1}{48}(\omega-1)^{6}+m+n\le\Tim\le4\o{\omega}^{4}\cdot(m+\omega n),\\
\frac{1}{8}(\omega-1)^{4}+m+n\le\Mem\le2\o{\omega}^{2}\cdot(m+\omega n),
\end{gather*}
in which $\omega\equiv\omega(E)$ and $\o{\omega}\equiv\omega(\o E)$
satisfy $1\le\omega\le\o{\omega}$.
\end{corollary}
\begin{proof}
Let $a_{i}\equiv\svec_{F}(A_{i})$ and $\omega_{j}\equiv|\col_{F}(j)|$.
We break the solution of $\Hess\Delta y=g$ into five steps:
\begin{enumerate}
\item (Input) It takes $\Mem_{\mathrm{data}}$ memory to state the problem
data, where $m+n\le\Mem_{\mathrm{data}}\le2\o{\omega}^{2}m+3\omega^{2}n$.
Indeed, $m+n\le\nnz(\A)\le\o{\omega}^{2}m+\omega^{2}n$ because $1\le\nnz(a_{i})\le\o{\omega}^{2}$
(see Step 2 below) and $\nnz(\P_{j})=\frac{1}{2}\omega_{j}(\omega_{j}+1)$.
Also, $m+n\le\nnz(w)\le m+\omega^{2}n$, and $n\le\nnz(g)\le\omega n$.
\item (Build LP part) It takes $\Tim_{\mathrm{LP}}$ time to build $\sum_{i=1}^{m}d_{i}a_{i}a_{i}^{T}$,
where $m\le\Tim_{\mathrm{LP}}\le2\o{\omega}^{4}m$ time and memory.
This follows from $1\le\nnz(a_{i})\le\o{\omega}^{2}$, where the upper-bound
is because $\spar(a_{i}a_{i}^{T})\subseteq E^{(2)}\subseteq\chol(E^{(2)})$,
and that $\omega(E^{(2)})$ is, by definition, the maximum number
of nonzero elements in a single column of $\chol(E^{(2)})$.
\item (Build SDP part) It takes $\Tim_{\mathrm{SDP}}$ time to build $\sum_{j=1}^{n}\P_{j}\D_{j}\P_{j}^{T}$,
where $n\le\Tim_{\mathrm{SDP}}\le\omega^{4}n$ time and memory. This
follows from $\nnz(\P_{j})=\frac{1}{2}\omega_{j}(\omega_{j}+1)$,
which implies $\nnz(\P_{j}\D_{j}\P_{j}^{T})=\frac{1}{4}\omega_{j}^{2}(\omega_{j}+1)^{2}$
for a fully-dense $\D_{j}\succ0$. 
\item (Factorization) It takes $\Tim_{\mathrm{fact}}$ time and $\Mem_{\mathrm{fact}}$
memory to factor $\L=\chol(\Hess)$, where $\frac{1}{48}(\omega-1)^{6}+n\le\Tim_{\mathrm{fact}}\le\o{\omega}^{5}n$
and $\frac{1}{8}(\omega-1)^{4}+n\le\Mem_{\mathrm{fact}}\le\o{\omega}^{3}n$.
The matrix $\Hess$ has $|F|$ columns and rows, and frontsize $\omega(E^{(2)})$.
The desired figures follow by substituting $\frac{1}{2}\omega^{2}\le\omega(E^{(2)})\le\o{\omega}^{2}$
and $n\le|F|\le\omega n$ into \propref{chol}.
\item (Back-substitution) It takes $\Mem_{\mathrm{fact}}$ time and memory
to solve each of $\L z=g$ and $\L^{T}\Delta y=z$ via triangular
back-substitution. 
\end{enumerate}
The overall runtime is cumulative, so $\Tim=\Tim_{\mathrm{LP}}+\Tim_{\mathrm{SDP}}+\Tim_{\mathrm{fact}}+2\Mem_{\mathrm{fact}}$.
The overall memory use is $\Mem=\Mem_{\mathrm{data}}+\Mem_{\mathrm{fact}}$,
because the matrix $\Hess$ can be constructed and then factored in-place.
\qed
\end{proof}

Let us now give an end-to-end complexity guarantee for \algref{myalg}.
We will need the following assumption to ensure that the data $(\A,\b,\c,\K)$
previously defined in (\ref{eq:AKdef}) specifies an SDP in $(N,\omega)$-standard
form, where $N=m+\omega n$ and $\omega\equiv\omega(E)$. 

\begin{assume}[Strong duality is attained]\label{asm:strict}There
exists a primal-dual pair $X^{\star}\succeq0$ and $v^{\star}\le0$
that are feasible $\inner{A_{i}}{X^{\star}}\le b_{i}$ for all $i$
and $\sum_{i}v_{i}^{\star}A_{i}\preceq C$ and coincide in their objectives
$\inner C{X^{\star}}=\inner b{v^{\star}}$.

\end{assume}
\begin{theorem}[Upper complexity]
\label{thm:total}Let the data $C,A_{1},\dots,A_{m}\in\S^{n}$ and
$b\in\R^{m}$ satisfy \asmref{strict}. Given a tree decomposition
of width $\o{\tau}$ for the extended aggregate sparsity graph $\o G=(V,\o E)$,
where
\[
V=\{1,2,\dots,n\},\qquad\o E=\spar(C)\cup\clique(A_{1})\cup\cdots\cup\clique(A_{m}),
\]
set $\Pi$ as the associated perfect elimination ordering. Then, \algref{myalg}
outputs $U\in\R^{n\times\o{\omega}}$ and $v\in\R^{m}$ with $v\le0$
such that
\begin{align*}
\inner{A_{i}}{UU^{T}}-b_{i} & \le\epsilon\text{ for all }i, & \sum_{i=1}^{m}v_{i}A_{i}-C & \preceq\epsilon\cdot I, & \inner C{UU^{T}}-\inner bv & \le\epsilon\cdot N,
\end{align*}
in $O(\sqrt{N}\log(1/\epsilon))$ iterations, with per-iteration costs
of $O(\o{\omega}^{4}\cdot N)$ time and $O(\o{\omega}^{2}\cdot N)$
memory, where $N=m+\o{\omega}\cdot n$ and $\o{\omega}=1+\o{\tau}$.
\end{theorem}
\begin{proof}One can verify that $(\A,\b,\c,\K)$ defined in (\ref{eq:AKdef})
specifies an SDP in $(N,\omega)$-standard form (see \appref{regularity}
for the regularity conditions). Moreover, it follows from the monotonicity
of the frontsize (\propref{del}) that $\omega\equiv\omega(E_{\Pi})\le\omega(\o E_{\Pi})\equiv\o{\omega}$.
We will track the cost of \algref{myalg} step-by-step:
\begin{enumerate}
\item (Front-reducing permutation) Preordering $A_{i}\gets\Pi A_{i}\Pi^{T}$
and $C\gets\Pi C\Pi^{T}$ the matrices cost $O(\nnz(C)+\sum_{i=1}^{m}\nnz(A_{i}))=O(\o{\omega}^{2}N)$
time and memory. This follows from $\nnz(C)\le|F|\le N$ and $\nnz(A_{i})\le\o{\omega}^{2}$. 
\item (Conversion) Computing $F=\chol(E)$ costs $O(|F|)=O(N)$ time and
space, where we note that $|F|\le\omega n$. 
\item (Solution) Let $K=2\cdot\max\{1,|\tr(C)|,|\tr(A_{1})|,\dots,|\tr(A_{m})|\}$.
After $O(\sqrt{N}\log(K/\epsilon))$ iterations, we arrive at a primal
$v\le0$ and $V_{j}\succeq0$ and dual point $Y\in\S_{F}^{n}$ satisfying
\begin{gather*}
\inner{A_{i}}Y-b_{i}\le\epsilon/K,\quad Y[J_{j},J_{j}]\succeq-(\epsilon/K)I,\\
{\textstyle \left\Vert \sum_{i}v_{i}A_{i}+\sum_{j}P_{j}V_{j}P_{j}^{T}-C\right\Vert \le\epsilon/K,}\\
\inner CY-\inner bv\le N\cdot(\epsilon/K).
\end{gather*}
Each iteration costs $O(\o{\omega}^{4}N)$ time and $O(\o{\omega}^{2}N)$
memory. This cost is fully determined by the cost of solving $O(1)$
instances of the Schur complement equation, which dominates the overhead
of $O(\omega^{3}n+\nnz(\A))=O(\o{\omega}^{2}N)$ time and memory.
\item (Recovery) Using the previously recovered $Y$, we recover $U$ such
that $\Pi_{F}(UU^{T})=Y+\delta I$ where $\delta=-\min_{j}\{0,\lambda_{\min}(Y[J_{j},J_{j}])\}\le\epsilon/K$.
This takes $O(\omega^{3}n)=O(\o{\omega}^{2}N)$ time and $O(\omega^{2}n)=O(\o{\omega}N)$
memory.
\item (Output) We output $U$ and $v$, and check for accuracy. It follows
from $K\ge2|\tr(C)|$ that
\[
\inner C{UU^{T}}-\inner CY=\delta\cdot\tr(C)\le\epsilon\cdot\frac{\tr(C)}{K}\le\epsilon\cdot\frac{|\tr(C)|}{2|\tr(C)|}\le\frac{1}{2}\epsilon,
\]
and from $K\ge2$ and $N\ge1$ that 
\[
\inner C{UU^{T}}-\inner bv\le\frac{N}{K}\epsilon+\frac{1}{2}\epsilon=N\epsilon\cdot\left(\frac{1}{2}+\frac{1}{2N}\right)\le N\cdot\epsilon.
\]
Similarly, it follows from $K\ge2|\tr(A_{i})|$ and $K\ge2$ that
$\inner{A_{i}}{UU^{T}}-b_{i}\le\epsilon$. Finally, it follows from
$v\ge0$ and $V_{j}\succeq0$ that
\[
\sum_{i}v_{i}A_{i}+\sum_{j}P_{j}V_{j}P_{j}^{T}-C\preceq\epsilon\cdot I\implies\sum_{i}v_{i}A_{i}-C\preceq\epsilon\cdot I.
\]
\end{enumerate}
\qed \end{proof}

Let $\o{\tau}_{\star}\equiv\tw(\o G)$. In theory, it takes $O(\o{\tau}_{\star}^{8}\cdot n\log n)$
time to compute a tree decomposition of width $\o{\tau}=O(\o{\tau}_{\star}^{2})$
by exhaustively enumerating the algorithm of \citet{fomin2018fully}.
Using this tree decomposition, \thmref{total} says that \algref{myalg}
arrives at an $\epsilon$-accurate solution in $O((m+n)^{1/2}\cdot\o{\tau}_{\star}\cdot\log(1/\epsilon))$
iterations, with per-iteration costs of $O((m+n)\cdot\o{\tau}_{\star}^{10})$
time and $O((m+n)\cdot\o{\tau}_{\star}^{6})$ memory. Combined, the
end-to-end complexity of solving (\ref{eq:sdp}) using \algref{myalg}
is $O(\o{\tau}_{\star}^{11}\cdot(m+n)^{1.5}\cdot\log(1/\epsilon))$
time. 

In practice, chordal conversion works even better. In \secref{exp},
we provide detailed numerical experiments to validate that: (i) the
minimum degree heuristic usually finds $\Pi$ that yield $\o{\omega}=O(1)$;
(ii) a primal-dual interior-point method usually converges to $\epsilon$
accuracy in dimension-free $O(\log(1/\epsilon))$ iterations. Taking
these as formal assumptions improves \thmref{total} to $O((m+n)\log(1/\epsilon))$
empirical time. 

The following establishes the sharpness of \thmref{total}.
\begin{corollary}[Lower complexity]
\label{cor:lb}Given the data $C,A_{1},\dots,A_{m}\in\S^{n}$ and
$b\in\R^{m}$, let $\tau_{\star}$ denote the treewidth of the aggregate
sparsity graph $G=(V,E)$:
\[
V=\{1,2,\dots,n\},\qquad E=\spar(C)\cup\spar(A_{1})\cup\cdots\cup\spar(A_{m}).
\]
There exists no choice of $\Pi$ that will allow \algref{myalg} to
solve (\ref{eq:sdp}) to arbitrary accuracy $\epsilon>0$ in less
than $\Omega(\tau_{\star}^{6}+m+n)$ time and $\Omega(\tau_{\star}^{4}+m+n)$
memory.
\end{corollary}
\begin{proof}
The cost of \algref{myalg} is at least a single iteration of the
interior-point method in Step 2. This is no less than $\Omega((\omega-1)^{6}+m+n)$
time and $\Omega((\omega-1)^{4}+m+n)$ memory according to \corref{normal},
where $\omega-1\ge\tau_{\star}$ due to \corref{treewidth}. \qed
\end{proof}

\section{\label{sec:per-iter}Frontsize of the Schur complement sparsity (Proof
of \thmref{front})}

We now turn to prove the frontsize bound on the Schur complement
sparsity $E^{(2)}$ in \thmref{front}, which we identified as our
key technical contribution. Recall that a symmetric sparsity pattern
$E$ of order $n$ can be viewed as the edge set of an undirected
graph $G=(V,E)$ on vertices $V=\{1,2,\dots,n\}$. The underlying
principle behind our proof is the fact that the frontsize is monotone
under the subgraph relation: if $G'=(V',E')$ is a subgraph of $G=(V,E)$,
then $\omega(E)\ge\omega(E')$. 

To state this formally, we denote the sparsity pattern induced by
a subset of vertices $U=\{u_{1},u_{2},\dots,u_{p}\}\subseteq V$ as
follows
\[
E[U]\equiv\{(i,j):(u_{i},u_{j})\in E\}\text{ where }u_{1}<u_{2}<\cdots<u_{p}.
\]
Note that we always sort the elements of $U$. Our definition is made
so that if $E=\spar(X)$, then $E[U]=\spar(X[U,U])$, without any
reordering of the rows and columns. 
\begin{proposition}[Subgraph monotonicity]
\label{prop:del}Let $E$ be a sparsity pattern of order $n$, and
let $U\subseteq\{1,2,\dots,n\}$. Then, for any sparsity pattern $D$
of order $|U|$ that satisfies $E[U]\supseteq D$, we have $\chol(E)[U]\supseteq\chol(D)$,
and therefore $\omega(E)\ge\omega(D)$.
\end{proposition}
\begin{proof}
It is known that $(i,j)\in\chol(E)$ for $i>j$ holds if and only
if there exists a path $(i,p_{1},p_{2},\dots,p_{\ell},j)$ whose edges
are in $E$, and whose internal nodes are ordered $p_{1},p_{2},\dots,p_{\ell}<j<i$;
see e.g.~\citep[Theorem~6.1]{vandenberghe2015chordal}. It immediately
follows this characterization that $\chol(\cdot)$ is monotone with
respect to the deletion of edges and isolated vertices: (1) if $D\subseteq E$,
then $\chol(D)\subseteq\chol(E)$; (2) we have $\chol(E[U])=\chol(E)[U]$
for $U=\{1,2,\dots,n\}\backslash v$ with isolated vertex $v$. Therefore,
$\chol(\cdot)$ must also be monotone under general vertex and edge
deletions, because we can always delete edges to isolate a vertex
before deleting it. \qed
\end{proof}

Our lower-bound is a direct corollary of the following result, which
gives an exact value for the frontsize of a certain ``lifted'' sparsity
pattern. 
\begin{lemma}[Quadratic lift]
\label{lem:quadlift}Let $E$ be an arbitrary sparsity pattern of
order $n$. Define $F=\chol(E)$ and the lifted sparsity pattern
$F^{(2)}=\bigcup_{k=1}^{n}\clique(\P_{k})$ in which each $\P_{k}$
is implicitly defined to satisfy $\P_{k}^{T}\svec_{F}(Y)=\svec(Y[\col_{F}(k),\col_{F}(k)])$
for all $Y\in\S_{F}^{n}$. Then, we have $\chol(F^{(2)})=F^{(2)}$
and $\omega(F^{(2)})=\frac{1}{2}\omega(E)[\omega(E)+1]$.
\end{lemma}
Observe that $E^{(2)}\supseteq F^{(2)}=\bigcup_{k=1}^{n}\clique(\P_{k})$
via (\ref{eq:E2def}), so it follows immediately from \propref{del}
\[
\omega(E^{(2)})\ge\omega(F^{(2)})=\frac{1}{2}\omega(E)[\omega(E)+1],
\]
which is precisely the lower-bound in \thmref{front}. For the upper-bound,
we will use $\o F=\chol(\o E)$, the symbolic Cholesky factor of the
extended aggregate sparsity pattern $\o E$, to construct a similarly
lifted $\o F^{(2)}$. Our key insight is that $E^{(2)}$ can be obtained
from $\o F^{(2)}$ via vertex and edge deletions.
\begin{lemma}[Sparsity overestimate]
\label{lem:ub}Let $\o E$ and $E^{(2)}$ be the sparsity patterns
defined in (\ref{eq:Ebardef}) and (\ref{eq:E2def}). Define $\o F=\chol(\o E)$
and the lifted sparsity pattern $\o F^{(2)}=\bigcup_{k=1}^{n}\clique(\o{\P}_{k})$
in which each $\o{\P}_{k}$ is implicitly defined to satisfy $\o{\P}_{k}^{T}\svec_{\o F}(Y)=\svec(Y[\col_{\o F}(k),\col_{\o F}(k)])$
for all $Y\in\S_{\o F}^{n}$. Then, $E^{(2)}\subseteq\o F^{(2)}[V^{(2)}]$
holds for $V^{(2)}=\{\idx_{\overline{F}}(i,j):i,j\in F\}$.
\end{lemma}
Substituting $E^{(2)}\subseteq\o F^{(2)}[V^{(2)}]$ with the exact
frontsize of $\omega(\o F^{(2)})$ from \lemref{quadlift} yields
\[
\omega(E^{(2)})\le\omega(\o F^{(2)})=\frac{1}{2}\omega(\o E)[\omega(\o E)+1],
\]
which is precisely the upper-bound in \thmref{front}. Finally, if
$E=\o E$, then $F^{(2)}=\o F^{(2)}=E^{(2)}$, so $\chol(E^{(2)})=E^{(2)}$
via \lemref{quadlift}.

In the remainder of this section, we will prove \lemref{quadlift}
and \lemref{ub}.

\subsection{Exact frontsize of a lifted sparsity pattern (Proof of \lemref{quadlift})}

Our proof of \lemref{quadlift} is based on a connection between
zero-fill sparsity patterns, for which sequential Gaussian elimination
results in no additionally fill-in, and a ``sorted'' extension of
the running intersection property. 
\begin{definition}[ZF]
The sparsity pattern $F$ is said to be\emph{ zero-fill} if $F=\chol(F)$.
\end{definition}
Equivalently, if $F$ is zero-fill, then $(i,k)\in F$ and $(j,k)\in F$
implies $(i,j)\in F$ for $i>j>k$ via the definition of the symbolic
Cholesky factor (\defref{symbchol}). 

\begin{definition}[RIP]
\label{def:rip}The sequence of subsets $J_{1},J_{2},\dots,J_{\ell}$
with $J_{j}\subset\N$ satisfies the \emph{sorted running intersection
property} if there exists a parent pointer $\p:\{1,2,\dots,\ell-1\}\to\{2,3,\dots,\ell\}$
such that the following holds for all $1\le j<\ell$:
\begin{gather*}
\p(j)>j,\quad J_{\p(j)}\supseteq J_{j}\cap(J_{j+1}\cup J_{j+2}\cup\cdots\cup J_{\ell}),\quad\min\{J_{\p(j)}\}>\max\{J_{j}\backslash J_{\p(j)}\}.
\end{gather*}
\end{definition}
The symbolic Cholesky factor $F=\chol(E)$ for a sparsity pattern
$E$ is the canonical example of a zero-fill sparsity pattern. In
turn, the corresponding column sets $J_{j}=\col_{F}(j)$ are the canonical
example of sequence of subsets that satisfy the sorted version of
the running intersection property.
\begin{proposition}[ZF$\implies$RIP]
\label{prop:zf->rip}Let $F$ be a zero-fill sparsity pattern of
order $n$. Then, the sequence of subsets $J_{1},J_{2},\dots,J_{n}$
with $J_{j}=\col_{F}(j)\equiv\{j\}\cup\{i>j:(i,j)\in F\}$ satisfies
the sorted running intersection property. 
\end{proposition}
\begin{proof}
Define $\p(j)=\min\{i>j:i\in J_{j}\}$ if $|J_{j}|>1$ and $\p(j)=n$
if $|J_{j}|=1$. Clearly, $\p(j)>j$ holds for all $1\le j<n$. To
prove $J_{\p(j)}\supseteq J_{j}\cap\bigcup_{w=j+1}^{\ell}J_{w}$,
let $i\in J_{j}\cap J_{w}$ for some $w>j$. We prove that $i\in J_{\p(j)}$
via the following steps: 
\begin{itemize}
\item We have $i>j$, because $i\in J_{w}$ implies that $i\ge w>j$. 
\item If $i=\p(j)$, then $i\in J_{\p(j)}$ by definition. 
\item If $i>\p(j)$, then $(i,j)\in F$ and $(\p(j),j)\in F$ for $i>\p(j)>j$
implies $(i,\p(j))\in F$, and hence $i\in J_{\p(j)}$. 
\end{itemize}
Finally, we prove $\min\{J_{\p(j)}\}>\max\{J_{j}\backslash J_{\p(j)}\}$
by noting that $\max\{J_{j}\backslash J_{\p(j)}\}=j$ with our construction,
and that $i=\min J_{\p(j)}$ must satisfy $i\in J_{\p(j)}$ and therefore
$i\ge\p(j)>j$. \qed
\end{proof}

Our proof is based on the fact that the ``lifted'' sparsity pattern
$F^{(2)}$ can be constructed as $F^{(2)}=\bigcup_{k=1}^{n}\clique(J_{k}^{(2)})$
with respect to the following ``lifted'' index sets
\begin{gather}
J_{k}^{(2)}\equiv\idx_{F}(\clique(J_{k}))=\{\idx_{F}(i,j):i,j\in J_{k},\quad i\ge j\}.\label{eq:Jk2_def}
\end{gather}
We need to show that, if the original index sets $J_{1},J_{2},\dots,J_{k}$
satisfy the running intersection property, then the lifted index sets
$J_{1}^{(2)},J_{2}^{(2)},\dots,J_{k}^{(2)}$ will inherit the running
intersection property. Our key insight is that the index operator
$\idx_{F}$ implements a \emph{raster ordering}.
\begin{lemma}[Raster ordering]
The ordering $\idx_{F}:F\to\N$ satisfies the following, for all
$(i,j)\in F$ with $i\ge j$ and $(i',j')\in F$ with $i'\ge j'$:
\begin{itemize}
\item If $j>j'$, then $\idx(i,j)>\idx(i',j')$ holds. 
\item If $j=j'$ and $i>i'$, then $\idx(i,j)>\idx(i',j')$ holds. 
\end{itemize}
\end{lemma}
\begin{lemma}
\label{lem:lift}Let the sequence of subsets $J_{1},J_{2},\dots,J_{\ell}$
with $J_{j}\subset\N$ satisfy the sorted running intersection property.
Then, $J_{1}^{(2)},J_{2}^{(2)},\dots,J_{\ell}^{(2)}$ also satisfy
the sorted running intersection property. 
\end{lemma}
\begin{proof}
Let $\p(\cdot)$ denote the parent pointer that verifies the sorted
running intersection property in $J_{1},J_{2},\dots,J_{\ell}$. We
will verify that $\p(\cdot)$ also proves the same property in $J_{1}^{(2)},J_{2}^{(2)},\dots,J_{\ell}^{(2)}$. 

First, to prove $J_{u}^{(2)}\cap\bigcup_{v=u+1}^{\ell}J_{v}^{(2)}\subseteq J_{\p(u)}^{(2)},$
let $k\in J_{u}^{(2)}\cap J_{v}^{(2)}$ for $v>u$. The fact that
$k\in J_{u}^{(2)}=\idx(\clique(J_{u}))$ implies $k=\idx(i,j)$ for
some $i,j$ such that $i,j\in J_{u}$. Similarly, $k\in J_{v}^{(2)}=\idx(\clique(J_{v}))$
and the bijectivity of $\idx$ on $F$ imply that the same $i,j$
also satisfy $i,j\in J_{v}$, where we recall that $v>u$. We conclude
$i,j\in J_{u}\cap J_{v}\subseteq J_{\p(u)}$ and therefore $k=\idx(i,j)\in J_{\p(u)}^{(2)}$.

Next, we prove $\min\{J_{\p(u)}^{(2)}\}>\max\{J_{u}^{(2)}\backslash J_{\p(u)}^{(2)}\}$
by establishing two claims:
\begin{itemize}
\item $\min\{J_{p(u)}^{(2)}\}=\idx(\alpha,\alpha)$ where $\alpha=\min J_{\p(u)}$.
For any $\idx(i,j)\in J_{\p(u)}^{(2)}$ where $i\ge j$, we must have
$i,j\in J_{\p(u)}$. It follows from the raster property that $\idx(i,j)$
is minimized with $i=j=\min J_{\p(u)}$.
\item $\max\{J_{u}^{(2)}\backslash J_{\p(u)}^{(2)}\}=\idx(\beta,\gamma)$
where $\beta=\max\{J_{u}\}$ and $\gamma=\max\{J_{u}\backslash J_{\p(u)}\}$.
We partition $J_{u}$ into $N_{u}=J_{u}\backslash J_{\p(u)}$ and
$A_{u}=J_{u}\cap J_{\p(u)}$. 
\begin{itemize}
\item For any $\idx(i,j)\in J_{u}^{(2)}\backslash J_{p(u)}^{(2)}$ where
$i\ge j$, we can have one of the following three cases: 1) $i\in N_{u}$
and $j\in A_{u}$; 2) $i\in A_{u}$ and $j\in N_{u}$; or 3) $i\in N_{u}$
and $j\in N_{u}$. 
\item We observe that the first case $i\in N_{u}$ and $j\in A_{u}$ is
impossible. Indeed, applying $j\ge\min\{J_{\p(u)}\}>\max\{J_{u}\backslash J_{\p(u)}\}\ge i$
would yield a contradiction with $i\ge j$. 
\item Taking the union of the two remaining cases yields $i\in J_{u}=A_{u}\cup N_{u}$
and $j\in N_{u}$. It follows from the raster property that $\idx(i,j)$
is maximized with $i=\max J_{u}$ and $j=\max N_{u}$.
\end{itemize}
\end{itemize}
With the two claims established, the hypothesis that $\min\{J_{\p(u)}\}>\max\{J_{u}\backslash J_{\p(u)}\}$
implies that $\alpha>\gamma$, and therefore $\idx(\alpha,\alpha)>\idx(\beta,\gamma)$
as desired. \qed
\end{proof}

In reverse, a sequence of subsets $J_{1}^{(2)},J_{2}^{(2)},\dots,J_{\ell}^{(2)}$
that satisfy the sorted running intersection property immediately
give rise to a zero-fill sparsity pattern $F^{(2)}=\bigcup_{k=1}^{\ell}\clique(J_{k}^{(2)})$
with $\omega(F^{(2)})=\max_{j}|J_{j}^{(2)}|$. 
\begin{proposition}[RIP$\implies$ZF]
\label{prop:rip->zf}Let $J_{1},J_{2},\dots,J_{\ell}$ with $\bigcup_{j=1}^{\ell}J_{j}=\{1,2,\dots,n\}$
satisfy the sorted running intersection property. Then, $F=\bigcup_{j=1}^{\ell}\clique(J_{j})$
is zero-fill, and we have $\omega(F)=\max_{j}|J_{j}|$.
\end{proposition}
Our proof of \propref{rip->zf} relies on the following result, which
says that every column of $F$ is contain in a subset $J_{w}$. 
\begin{lemma}
\label{lem:colrip}Let $J_{1},J_{2},\dots,J_{\ell}$ with $\bigcup_{j=1}^{\ell}J_{j}=\{1,2,\dots,n\}$
satisfy the sorted running intersection property. For every $j$-th
column in $F=\bigcup_{j=1}^{\ell}\clique(J_{j})$, there exists some
$J_{w}$ such that $\col_{F}(j)\subseteq J_{w}.$
\end{lemma}
\begin{proof}
Let $V=\{1,2,\dots,n\}$. For arbitrary $j\in V$, denote $J_{w}$
as the last subset in the sequence $J_{1},J_{2},\dots,J_{\ell}$ for
which $j\in J_{w}$. This choice must exist, because $V=\bigcup_{k=1}^{\ell}J_{k}$.
For every arbitrary $i\in\col_{F}(j)$, we will prove that $i\in J_{w}$
also holds:
\begin{itemize}
\item There exists $u\le w$ for which $i,j\in J_{u}$. Indeed, $(i,j)\in F$
and $F=\bigcup_{k=1}^{\ell}\clique(J_{k})$ imply $(i,j)\in\clique(J_{u})$
for some $J_{u}$, or equivalently $i,j\in J_{u}$. Given that $j\in J_{u}$
and $w=\max\{k:j\in J_{k}\}$ by definition, it follows that $u\le w$.
\item If $u=w$, then $i\in J_{w}$ holds because $i,j\in J_{u}$. Otherwise,
if $u<w$, we use the sorted running intersection property to assert
the following
\begin{gather}
u<w,\qquad i,j\in J_{u}\quad\implies\quad i,j\in J_{\p(u)}.\label{eq:par1}
\end{gather}
The fact that $j\in J_{\p(u)}$ follows directly the running intersection
property $j\in J_{u}\cap J_{w}\subseteq J_{\p(u)}$ for $w>u$. By
contradiction, suppose that $i\notin J_{\p(u)}$. Then, given that
$i\in J_{u}$, it follows from the sorted property that $j\ge\min\{J_{\p(u)}\}>\max\{J_{u}\backslash J_{\p(u)}\}\ge i,$
which contradicts our initial hypothesis that $i>j$. 
\item Inductively reapplying (\ref{eq:par1}), as in $i,j\in J_{\p(\p(u))}$
and $i,j\in J_{\p(\p(\p(u)))}$, we arrive at some $v$ such that
$i,j\in J_{\p(v)}$ and $p(v)=w$. The induction must terminate with
$p(v)\ge w$ because each $\p(u)>u$ by construction. It is impossible
for $\p(v)>w$ to occur, because $j\in J_{\p(v)}$ and $w=\max\{k:j\in J_{k}\}$
by definition. We conclude that $i\in J_{w}$, as desired.
\end{itemize}
\qed
\end{proof}

The equivalence between the sorted running intersection property and
zero-fill sparsity pattern then follow as a short corollary of the
above.

\begin{proof}[Proof of \propref{rip->zf}]To prove that $F$ is zero-fill,
we observe, for arbitrary $(i,k)\in F$ and $(j,k)\in F$ with $i>j>k$
that $i,j\in\col_{F}(k)$. Therefore, it follows from \lemref{colrip}
that there exists $J_{w}$ such that $i,j,k\in J_{w}$. We conclude
that $(i,j)\in F,$ because $F=\bigcup_{j=1}^{\ell}\clique(J_{j})$
and $i,j\in J_{w}$ for some $1\le w\le\ell$. 

To prove that $\omega(F)=\max_{j}|J_{j}|$, we choose $k=\arg\max_{j}|\col_{F}(j)|$.
It follows from Lemma~\ref{lem:colrip} that there exists $J_{w}$
such that $\col_{F}(k)\subseteq J_{w}$, and therefore $\omega(F)=|\col_{F}(k)|\le|J_{w}|$.
Finally, given that $\clique(J_{w})\subseteq F$ it follows that $(i,j)\in F$
holds for all $i\in J_{w}$ where $j=\min J_{w}$. Therefore, we conclude
that $J_{w}\subseteq\col_{F}(j)$, and therefore $|J_{w}|\le|\col_{F}(j)|\le|\col_{F}(k)|=\omega(F)$.
\qed \end{proof}

Finally, we conclude the proof of \lemref{quadlift} by verifying
that a matrix like $\Hess=\sum_{k=1}^{n}\P_{k}\D_{k}\P_{k}^{T}$ does
indeed have $F^{(2)}=\bigcup_{k=1}^{n}\clique(J_{k}^{(2)})$ as its
sparsity pattern.

\begin{proof}[Proof of \lemref{quadlift}]For an arbitrary sparsity
pattern $E$ of order $n$, let $F=\chol(E)$ and $J_{k}\equiv\col_{F}(k)$.
First, it follows from \propref{zf->rip} that $J_{1},J_{2},\dots,J_{n}$
satisfy the ordered running intersection property. Therefore, it follows
from \lemref{lift}, $J_{1}^{(2)},J_{2}^{(2)},\dots,J_{n}^{(2)}$
defined in (\ref{eq:Jk2_def}) satisfy the same property. Finally,
we verify that 
\begin{align*}
\supp(\P_{k})\equiv & \left\{ \alpha:\P_{k}[\alpha,\beta]\ne0\right\} =\left\{ \alpha:\P_{k}^{T}e_{\alpha}\ne0\right\} \\
\overset{\text{(a)}}{=} & \left\{ \idx_{F}(i,j):\P_{k}^{T}\svec_{F}(e_{i}e_{j}^{T})\ne0,\;(i,j)\in F\right\} \\
\overset{\text{(b)}}{=} & \left\{ \idx_{F}(i,j):(e_{i}e_{j}^{T})[J_{k},J_{k}]\ne0,\;i\ge j\right\} =J_{k}^{(2)}.
\end{align*}
Equality (a) is obtained by substituting $e_{\alpha}=\svec_{F}(e_{i}e_{j}^{T})$
for $(i,j)\in F$. Equality (b) follows the identity $\P_{k}^{T}\svec_{F}(Y)=\svec(Y[J_{k},J_{k}])$,
which was used to define $\P_{k}$. Therefore, $F^{(2)}=\bigcup_{k=1}^{n}J_{k}^{(2)}$
and we conclude via \propref{rip->zf} that $\omega(F^{(2)})=\max_{k}|J_{k}^{(2)}|=\max_{k}\frac{1}{2}|J_{k}|(|J_{k}|+1)$,
and we recall that $\max_{k}|J_{k}|=\omega(F)=\omega(E)$ by definition.
\qed \end{proof}

\subsection{Sparsity overestimate (Proof of \lemref{ub})}

We will need some additional notation. For $V=\{1,2,\dots,n\}$ and
$J\subseteq V$, we denote the subset of $J$ induced by the elements
in $U=\{u_{1},u_{2},\dots,u_{p}\}\subseteq V$ as follows
\[
J[U]\equiv\{i:u_{i}\in J\}\text{ where }u_{1}<u_{2}<\cdots<u_{p}.
\]
Our definition is made so that if $E=\clique(J)$, then $E[U]=\clique(J[U])$.
Our desired claim, namely that 
\begin{align*}
E^{(2)} & =\left(\bigcup_{i=1}^{m}\clique(\supp(a_{i}))\right)\cup\left(\bigcup_{j=1}^{n}\clique(\supp(\P_{j}))\right)\\
 & \subseteq\left(\bigcup_{k=1}^{n}\clique(\supp(\o{\P}_{k})[V^{(2)}])\right)=\o F^{(2)}[V^{(2)}]
\end{align*}
where $V^{(2)}=\{\idx_{\overline{F}}(i,j):i,j\in F\}$, now follows
immediately from the following two lemmas.
\begin{lemma}
For every $i\in\{1,2,\dots,m\}$, there exists $k\in\{1,2,\dots,n\}$
such that $\supp(a_{i})\subseteq\supp(\o{\P}_{k})[V^{(2)}]$ where
$V^{(2)}=\{\idx_{\overline{F}}(i,j):i,j\in F\}$. 
\end{lemma}
\begin{proof}
It follows by repeating the proof of \lemref{quadlift} that
\begin{gather*}
\supp(\o{\P}_{k})=\{\idx_{\o F}(i,j):i,j\in\col_{\o F}(k),i\ge j\},\\
\supp(\o{\P}_{k})[V^{(2)}]=\{\idx_{F}(i,j):i,j\in\col_{\o F}(k),i\ge j\}.
\end{gather*}
Our desired claim follows via the following sequence of inclusions
\begin{align*}
\supp(a_{i}) & =\{v:a_{i}[v]\ne0\}=\{\idx_{F}(u,v):A_{i}[u,v]\ne0,\quad(u,v)\in F\}\\
 & \overset{\text{(a)}}{\subseteq}\{\idx_{F}(u,v):u,v\in\supp(A_{i}),\quad(u,v)\in F\}\\
 & \overset{\text{(b)}}{\subseteq}\{\idx_{F}(u,v):u,v\in\col_{\o F}(k),\quad(u,v)\in F\}=\supp(\o{\P}_{k})[V^{(2)}].
\end{align*}
Inclusion (a) is because $\spar(A_{i})\subseteq\clique(\supp(A_{i}))$.
Inclusion (b) is true via the following: If $J=\supp(A_{i})$ satisfies
$\clique(J)\subseteq\o F$, then $J\subseteq\col_{\o F}(k)$ where
$k=\min J$. Indeed, we have $k\in\col_{\o F}(k)$ by definition.
For any arbitrary $j\in J$ with $j>k$, we must have an edge $(j,k)\in\clique(J)\subseteq\o F$,
and therefore $j\in\col_{\o F}(k)$. \qed
\end{proof}

\begin{lemma}
For every $k\in\{1,2,\dots,n\}$, we have $\supp(\P_{k})\subseteq\supp(\o{\P}_{k})[V^{(2)}]$
where $V^{(2)}=\{\idx_{\overline{F}}(i,j):i,j\in F\}$. 
\end{lemma}
\begin{proof}
It again follows by repeating the proof of \lemref{quadlift} that
\begin{align*}
\supp(\P_{k}) & =\{\idx_{F}(i,j):i,j\in\col_{F}(k),i\ge j\}\\
 & \overset{\text{(a)}}{\subseteq}\{\idx_{F}(i,j):i,j\in\col_{\o F}(k),i\ge j\}=\supp(\o{\P}_{k})[V^{(2)}].
\end{align*}
Inclusion (a) is true because $E\subseteq\o E$ implies $F=\chol(E)\subseteq\chol(\o E)=\o F$
via \propref{del}, and therefore $\col_{F}(k)\subseteq\col_{\o F}(k)$.
\qed
\end{proof}

\section{\label{sec:exp}Large-scale numerical experiments}

Our goal in this section is to provide experimental evidence to justify
the following four empirical claims made in the paper:
\begin{enumerate}
\item Whenever a graph has small treewidth $\tau_{\star}=O(1)$, a fill-reducing
heuristic is also able to find a ``good enough'' tree decomposition
with width $\tau=O(1)$.
\item A primal-dual interior-point method consistently converges to high
accuracies of $\epsilon\approx10^{-6}$ in just tens of iterations,
at an essentially dimension-free rate.
\item In theory, clique-tree conversion (\ref{eq:ctc}) enjoys similar guarantees
to chordal conversion (\ref{eq:cc}). But in practice, (\ref{eq:cc})
is \emph{much} faster than (\ref{eq:ctc}), both in solution time
and in preprocessing time.
\item Real-world power systems $\G$ yield instances of the AC optimal power
flow relaxation (\exaref{acopf}) with small values of $\tw(\o G)=2\cdot\tw(\G^{2})$.
\end{enumerate}
To this end, we benchmark the following three conversion methods: 
\begin{itemize}
\item \textbf{CC:} Chordal conversion as outlined in this paper in \algref{myalg},
implemented in MATLAB, with MOSEK~\citep{mosek2019} as the general-purpose
solver. If $G=\o G$ has small treewidth, then CC is guaranteed to
use at most $O(m+n)$ time per-iteration via \thmref{total}.
\item \textbf{Dual CTC (heuristic):} The dualized variant of clique-tree
conversion of \citet{zhang2020sparse} based on the aggregate sparsity
graph $G$. We take MATLAB / MOSEK implementation directly from the
project website\footnote{\href{https://github.com/ryz-codes/dual_ctc}{https://github.com/ryz-codes/dual\_ctc}}.
If $G=\o G$ has small treewidth, then this variant is guaranteed
to use at most $O(m+n)$ time per-iteration via \citep[Theorem~1]{zhang2020sparse}.
When $G\ne\o G$, this variant reduces to an empirical heuristic. 
\item \textbf{Dual CTC (provable):} The dualized variant of clique-tree
conversion of \citet{zhang2020sparse}, but forced to use the extended
sparsity graph $\o G$ instead of $G$, as suggested by \citet{gu2022faster}.
It is implemented by padding the elements of the sparse cost matrix
$C$ with numerically-zero elements that are structurally nonzero,
so that it is recognized as an element of $\S_{\o E}^{n}$, and then
calling Dual CTC (heuristic). If $\o G$ has small treewidth, then
this variant is guaranteed to use at most $O(m+n)$ time per-iteration
via \citep[Theorem~1]{zhang2020sparse}.
\end{itemize}
All of our experiments were conducted on a modest workstation, with
a Xeon 3.3 GHz quad-core CPU and 32 GB of RAM. Our code was written
in MATLAB 9.8.0.1323502 (R2020a), and the general-purpose solver we
use is MOSEK v9.1.10~\citep{mosek2019}. MOSEK specifies default
parameters $\epsilon=10^{-8}$ and seek to terminate with $\|\A x-\b\|_{\infty}\le\epsilon(1+\|\b\|_{\infty})$
and $\|\A^{T}y+s-\c\|_{\infty}\le\epsilon(1+\|\c\|_{\infty})$ and
$\max\{x^{T}s,\c^{T}x-\b^{T}y\}\le\epsilon\max\{1,|\b^{T}x|,|\c^{T}x|\}$~\citep[Section 13.3.2]{mosek2019}.
If MOSEK is unable to achieve this accuracy due to numerical issues,
it gives up and accepts the solution as optimal if $\epsilon=10^{-5}$~\citep[Section 13.3.3]{mosek2019}.
Our calculation for the number of accurate digits is identical to
\citep[Section~8]{zhang2020sparse}, which was in turn adapted from
the DIMAC challenge~\citep{pataki2002dimacs}.

\subsection{\label{subsec:theta}Lovász theta problem on synthetic dataset}

\begin{table}
\caption{\label{tab:lovasz}Tree decomposition quality, accuracy (in decimal
digits) and timing (in seconds) for Lovász theta problems on partial
$k$-trees: $|\protect\V|$ - number of vertices; $|\protect\E|$
- number of edges; $\tau$ - width of tree decomposition used; \textquotedblleft prep\textquotedblright{}
- preprocessing time, which includes the conversion process and MOSEK's
internal preparation time; \textquotedblleft digit\textquotedblright{}
- accurate decimal digits; \textquotedblleft iter\textquotedblright{}
- number of interior-point iterations; \textquotedblleft per-it\textquotedblright{}
- average time per interior-point iteration.}

\resizebox{\textwidth}{!}{

\begin{tabular}{|c|c|ccccc|ccccc|ccccc|}
\hline 
 &  & \multicolumn{5}{c|}{Dual CTC ($k$-tree ordering)} & \multicolumn{5}{c|}{CC (amd ordering)} & \multicolumn{5}{c|}{CC ($k$-tree ordering)}\tabularnewline
\cline{3-17} \cline{4-17} \cline{5-17} \cline{6-17} \cline{7-17} \cline{8-17} \cline{9-17} \cline{10-17} \cline{11-17} \cline{12-17} \cline{13-17} \cline{14-17} \cline{15-17} \cline{16-17} \cline{17-17} 
$|\V|$ & $|\E|$ & $\tau$ & prep & digit & iter & per-it & $\tau$ & prep & digit & iter & per-it & $\tau$ & prep & digit & iter & per-it\tabularnewline
\hline 
100 & 126 & 16 & 0.05 & 7.0 & 8 & 0.04 & 9 & 0.02 & 6.8 & 12 & 0.01 & 16 & 0.15 & 7.9 & 10 & 0.01\tabularnewline
200 & 262 & 23 & 0.10 & 8.3 & 9 & 0.06 & 16 & 0.01 & 7.0 & 10 & 0.01 & 23 & 0.03 & 7.9 & 10 & 0.03\tabularnewline
500 & 684 & 33 & 0.28 & 8.8 & 9 & 0.12 & 30 & 0.03 & 6.7 & 9 & 0.06 & 33 & 0.07 & 7.1 & 7 & 0.10\tabularnewline
1,000 & 1,492 & 35 & 0.53 & 9.1 & 9 & 0.17 & 33 & 0.05 & 6.6 & 8 & 0.12 & 35 & 0.05 & 8.7 & 8 & 0.14\tabularnewline
2,000 & 3,004 & 35 & 0.85 & 7.3 & 11 & 0.14 & 35 & 0.09 & 6.5 & 8 & 0.15 & 35 & 0.09 & 6.8 & 8 & 0.15\tabularnewline
5,000 & 7,441 & 35 & 2.64 & 8.6 & 11 & 0.63 & 40 & 0.23 & 7.5 & 9 & 0.37 & 35 & 0.24 & 6.2 & 8 & 0.35\tabularnewline
10,000 & 14,977 & 35 & 6.73 & 7.3 & 15 & 1.60 & 36 & 0.48 & 6.7 & 8 & 0.81 & 35 & 0.47 & 8.0 & 9 & 0.67\tabularnewline
20,000 & 30,032 & 35 & 22.21 & 7.6 & 13 & 3.41 & 42 & 1.01 & 7.3 & 8 & 1.48 & 35 & 1.04 & 7.9 & 9 & 1.31\tabularnewline
50,000 & 75,042 & 35 & 114.21 & 3.8 & 20 & 9.13 & 44 & 2.91 & 5.1 & 8 & 3.81 & 35 & 2.79 & 6.7 & 10 & 3.27\tabularnewline
100,000 & 150,298 & 35 & 415.89 & 4.5 & 27 & 18.57 & 40 & 5.98 & 7.8 & 9 & 8.20 & 35 & 5.93 & 5.5 & 10 & 6.65\tabularnewline
200,000 & 299,407 & 35 & 1966.70 & 3.8 & 16 & 36.89 & 42 & 11.90 & 6.1 & 15 & 27.65 & 35 & 12.06 & 6.3 & 11 & 12.92\tabularnewline
500,000 & 749,285 & 35 & 12265.85 & 3.3 & 16 & 102.56 & 46 & 31.80 & 2.8 & 18 & 103.85 & 35 & 41.41 & 5.6 & 16 & 39.20\tabularnewline
1,000,000 & 1,500,873 & 35 & - & - & - & - & 46 & 65.39 & 4.4 & 18 & 414.80 & 35 & 67.05 & 3.3 & 17 & 97.57\tabularnewline
\hline 
\end{tabular}

}
\end{table}
\begin{figure}
\subfloat[]{\includegraphics[width=0.33\columnwidth]{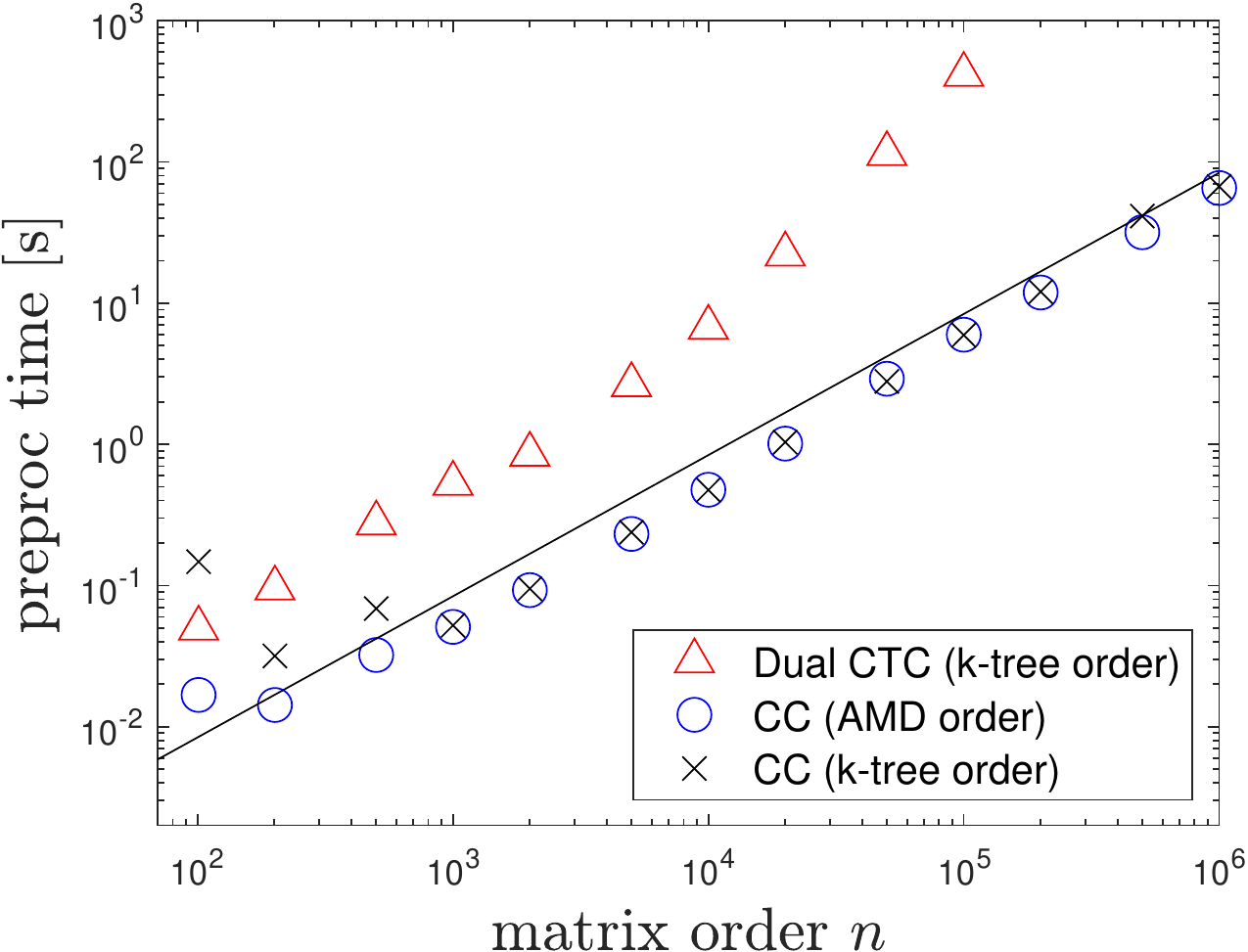}

}\subfloat[]{\includegraphics[width=0.33\columnwidth]{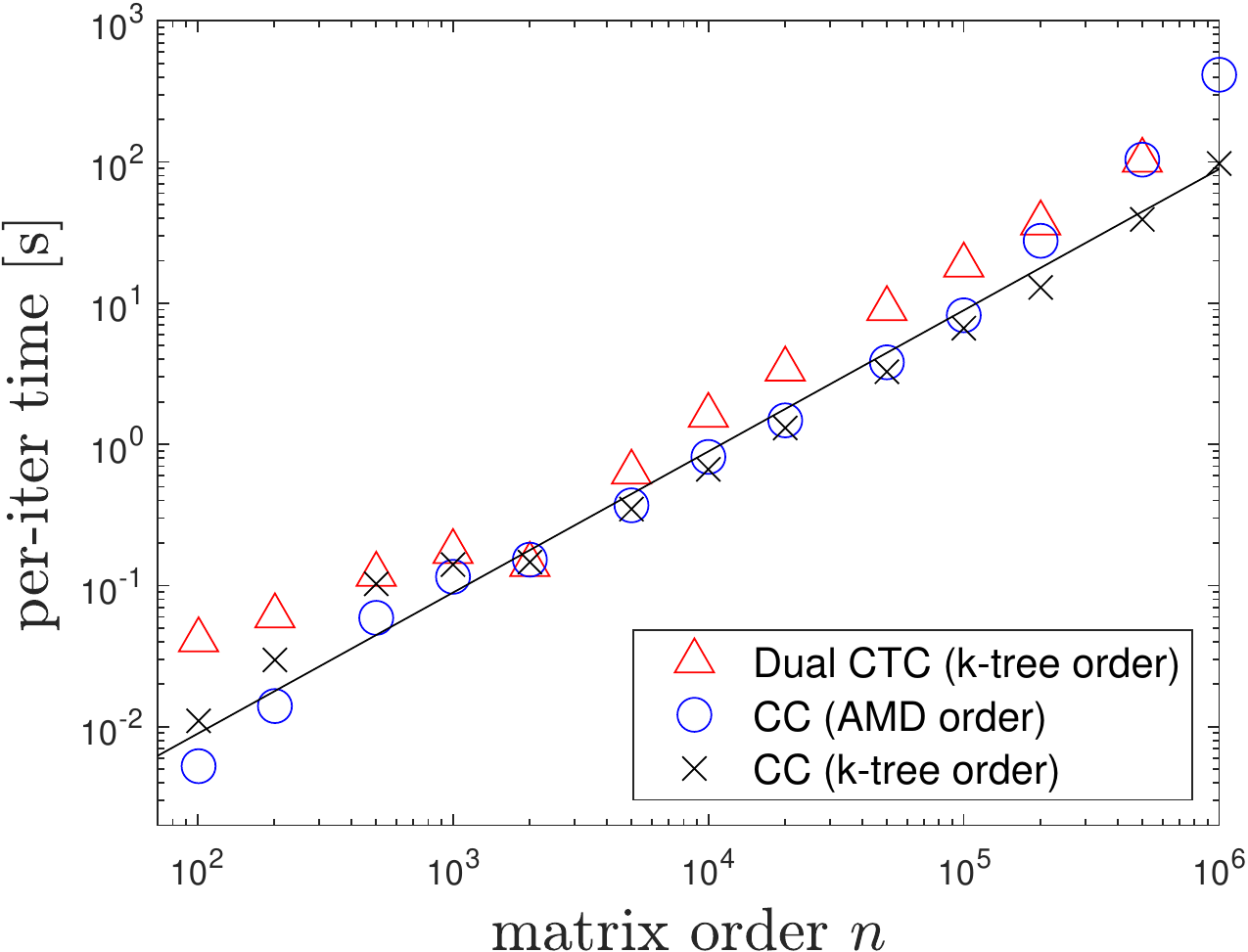}

}\subfloat[]{\includegraphics[width=0.33\columnwidth]{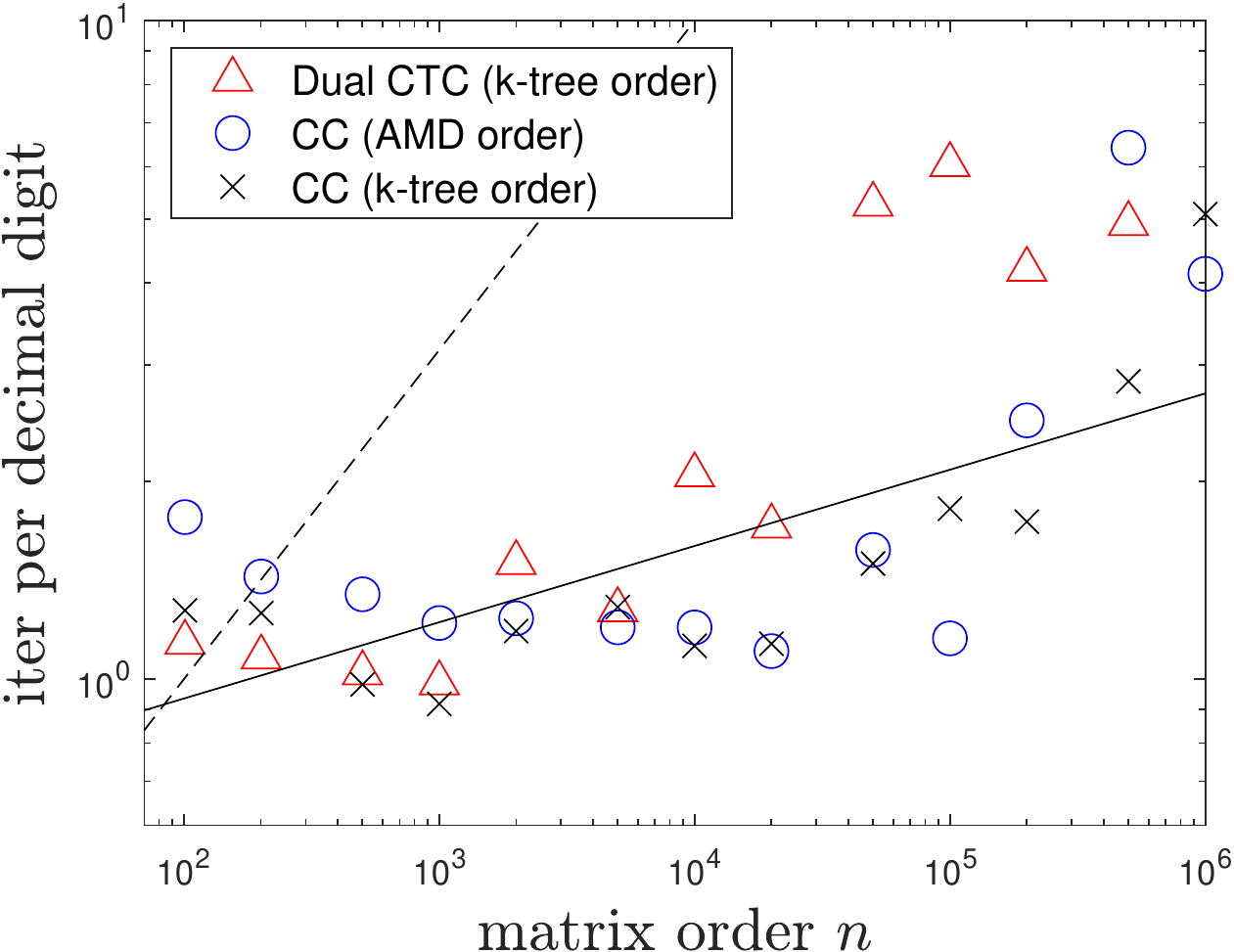}

}

\caption{\label{fig:theta_cc}Lovász theta problems solved via chordal conversion:
(a) Preprocessing time, with regression $p_{\times}(n)=8.385n\times10^{-5}$
and $R^{2}=0.87$; (b) Time per iteration, with regression $f_{\times}(n)=8.9n\times10^{-5}$
and $R^{2}=0.98$; (c) Iterations per decimal digit of accuracy, with
(solid) regression $g(n)=0.548n^{0.115}$ and $R^{2}=0.45$ and (dashed)
bound $g(n)=0.1\sqrt{n}$.}
\end{figure}

Our first set of experiments is on the Lovász theta problem (\exaref{theta}),
for which $G=\o G$ always holds with equality. For each trial, we
set $\G=(\V,\E)$ by randomly generating a $k$-tree with $k=35$
(see \citep[p.~9]{vandenberghe2015chordal} for details) and then
deleting edges uniformly at random until $|\E|/|\V|\approx3/2$. The
resulting $\G$ should have treewidth exactly $\tau_{\star}=35$ in
the limit $|\V|\to\infty$; the optimal ordering $\Pi$ to yield $\omega(E_{\Pi})=36$
is simply any perfect elimination ordering on the $k$-tree (``$k$-tree
ordering'' in \tabref{lovasz}). We observe that the \texttt{amd}
heuristic in MATLAB~\citep{amestoy2004algorithm} finds high-quality
orderings to yield $\omega(E_{\Pi})\le47$, corresponding to tree
decompositions of width $\tau\le46$, which is only about 30\% worse
than the best possible (``amd ordering'' in \tabref{lovasz}). Nevertheless,
minor differences in $\tau$ can still manifest as larger differences
in per-iteration time. 

We solve the problem on $\G$ using CC and CTC, and observe that in
both cases, it takes around 10 iterations to achieve $\epsilon\approx10^{-6}$
across a wide range of $n$, until numerical issues at very large
scales $n\approx10^{5}$ forced more iterations to be taken (see \figref{theta_cc}c
and \tabref{lovasz}). We find that both CC and CTC achieve comparable
$O(m+n)$ runtime per-iteration (see \figref{theta_cc}b), but CC
is significantly faster in its preprocessing time (see \figref{theta_cc}a).
As shown in the last few rows of \tabref{lovasz}, CC solved an instance
of the Lovász theta problem on a graph with $10^{6}$ vertices and
$1.5\times10^{6}$ edges in less than 30 minutes, taking a little
over 1 minute in the preprocessing. In contrast, CTC for a graph of
half this size took 3.5 hours just to perform the preprocessing.

To test the zero fill-in prediction in \thmref{front}, our implementation
of CC in this section forces MOSEK to factor its Schur complement
matrix $\Hess=\A\nabla^{2}f(w)\A^{T}$ \emph{without} the use of a
fill-reducing ordering, by setting the flag \texttt{MSK\_IPAR\_INTPNT\_ORDER\_METHOD}
to \texttt{'MSK\_ORDER\_METHOD\_NONE'}. If \thmref{front} is incorrect,
then factoring $\L=\chol(\Hess)$ would catastrophic dense fill-in,
and the per-iteration runtime would not be $O(m+n)$ as shown in \figref{theta_cc}b. 

\subsection{\label{subsec:acopf}AC optimal power flow relaxation on real-world
dataset}

\begin{table}
\caption{\label{tab:treedecomp}Treewidth bounds for the 72 power system test
cases in the MATPOWER dataset: $|V|$ - number of vertices; $|E|$
- number of edges; $\protect\lb,\protect\ub$ - lower-bound and upper-bounds
on $\protect\tw(\protect\G)$; $\protect\lb_{2},\protect\ub_{2}$
- lower-bound and upper-bounds on $\protect\tw(\protect\G^{2})$.}

\setlength\tabcolsep{3pt}
\resizebox{\textwidth}{!}{

\begin{tabular}{|c|ccccccc|c|ccccccc|}
\hline 
\# & Name & $|V|$ & $|E|$ & $\lb$ & $\ub$ & $\lb_{2}$ & $\ub_{2}$ & \# & Name & $|V|$ & $|E|$ & $\lb$ & $\ub$ & $\lb_{2}$ & $\ub_{2}$\tabularnewline
\hline 
1 & case4\_dist & 4 & 3 & 1 & 1 & 2 & 2 & 37 & case85 & 85 & 84 & 1 & 1 & 4 & 4\tabularnewline
2 & case4gs & 4 & 4 & 2 & 2 & 3 & 3 & 38 & case89pegase & 89 & 206 & 8 & 11 & 17 & 27\tabularnewline
3 & case5 & 5 & 6 & 2 & 2 & 4 & 4 & 39 & case94pi & 94 & 93 & 1 & 1 & 4 & 4\tabularnewline
4 & case6ww & 6 & 11 & 3 & 3 & 5 & 5 & 40 & case118zh & 118 & 117 & 1 & 1 & 4 & 4\tabularnewline
5 & case9target & 9 & 9 & 2 & 2 & 4 & 4 & 41 & case118 & 118 & 179 & 4 & 4 & 9 & 12\tabularnewline
6 & case9Q & 9 & 9 & 2 & 2 & 4 & 4 & 42 & case136ma & 136 & 135 & 1 & 1 & 8 & 8\tabularnewline
7 & case9 & 9 & 9 & 2 & 2 & 4 & 4 & 43 & case141 & 141 & 140 & 1 & 1 & 4 & 4\tabularnewline
8 & case10ba & 10 & 9 & 1 & 1 & 2 & 2 & 44 & case145 & 145 & 422 & 7 & 10 & 21 & 33\tabularnewline
9 & case12da & 12 & 11 & 1 & 1 & 2 & 2 & 45 & case\_ACTIVSg200 & 200 & 245 & 4 & 8 & 11 & 18\tabularnewline
10 & case14 & 14 & 20 & 2 & 2 & 6 & 6 & 46 & case300 & 300 & 409 & 3 & 6 & 11 & 17\tabularnewline
11 & case15nbr & 15 & 14 & 1 & 1 & 4 & 4 & 47 & case\_ACTIVSg500 & 500 & 584 & 4 & 8 & 14 & 22\tabularnewline
12 & case15da & 15 & 14 & 1 & 1 & 4 & 4 & 48 & case1354pegase & 1354 & 1710 & 5 & 12 & 13 & 30\tabularnewline
13 & case16am & 15 & 14 & 1 & 1 & 4 & 4 & 49 & case1888rte & 1888 & 2308 & 4 & 12 & 14 & 38\tabularnewline
14 & case16ci & 16 & 13 & 1 & 1 & 3 & 3 & 50 & case1951rte & 1951 & 2375 & 5 & 12 & 15 & 43\tabularnewline
15 & case17me & 17 & 16 & 1 & 1 & 3 & 3 & 51 & case\_ACTIVSg2000 & 2000 & 2667 & 6 & 40 & 16 & 85\tabularnewline
16 & case18nbr & 18 & 17 & 1 & 1 & 4 & 4 & 52 & case2383wp & 2383 & 2886 & 5 & 23 & 9 & 51\tabularnewline
17 & case18 & 18 & 17 & 1 & 1 & 3 & 3 & 53 & case2736sp & 2736 & 3263 & 5 & 23 & 9 & 55\tabularnewline
18 & case22 & 22 & 21 & 1 & 1 & 3 & 3 & 54 & case2737sop & 2737 & 3263 & 5 & 23 & 9 & 57\tabularnewline
19 & case24\_ieee\_rts & 24 & 34 & 3 & 4 & 7 & 8 & 55 & case2746wop & 2746 & 3299 & 5 & 23 & 9 & 53\tabularnewline
20 & case28da & 28 & 27 & 1 & 1 & 3 & 3 & 56 & case2746wp & 2746 & 3273 & 5 & 24 & 9 & 58\tabularnewline
21 & case30pwl & 30 & 41 & 3 & 3 & 7 & 9 & 57 & case2848rte & 2848 & 3442 & 5 & 18 & 14 & 41\tabularnewline
22 & case30Q & 30 & 41 & 3 & 3 & 7 & 9 & 58 & case2868rte & 2868 & 3471 & 5 & 17 & 16 & 43\tabularnewline
23 & case30 & 30 & 41 & 3 & 3 & 7 & 9 & 59 & case2869pegase & 2869 & 3968 & 9 & 12 & 17 & 42\tabularnewline
24 & case\_ieee30 & 30 & 41 & 3 & 3 & 7 & 9 & 60 & case3012wp & 3012 & 3566 & 5 & 25 & 10 & 55\tabularnewline
25 & case33bw & 33 & 32 & 1 & 1 & 3 & 3 & 61 & case3120sp & 3120 & 3684 & 5 & 28 & 9 & 60\tabularnewline
26 & case33mg & 33 & 32 & 1 & 1 & 3 & 3 & 62 & case3375wp & 3374 & 4068 & 6 & 27 & 12 & 64\tabularnewline
27 & case34sa & 34 & 33 & 1 & 1 & 3 & 3 & 63 & case6468rte & 6468 & 8065 & 5 & 26 & 15 & 65\tabularnewline
28 & case38si & 38 & 37 & 1 & 1 & 3 & 3 & 64 & case6470rte & 6470 & 8066 & 5 & 26 & 15 & 64\tabularnewline
29 & case39 & 39 & 46 & 3 & 3 & 5 & 7 & 65 & case6495rte & 6495 & 8084 & 5 & 26 & 15 & 62\tabularnewline
30 & case51ga & 51 & 50 & 1 & 1 & 3 & 3 & 66 & case6515rte & 6515 & 8104 & 5 & 26 & 16 & 62\tabularnewline
31 & case51he & 51 & 50 & 1 & 1 & 3 & 3 & 67 & case9241pegase & 9241 & 14207 & 21 & 33 & 42 & 78\tabularnewline
32 & case57 & 57 & 78 & 3 & 5 & 6 & 12 & 68 & case\_ACTIVSg10k & 10000 & 12217 & 5 & 33 & 17 & 80\tabularnewline
33 & case69 & 69 & 68 & 1 & 1 & 4 & 4 & 69 & case13659pegase & 13659 & 18625 & 21 & 31 & 42 & 80\tabularnewline
34 & case70da & 70 & 68 & 1 & 1 & 3 & 3 & 70 & case\_ACTIVSg25k & 25000 & 30110 & 6 & 51 & 17 & 127\tabularnewline
35 & case\_RTS\_GMLC & 73 & 108 & 4 & 5 & 7 & 11 & 71 & case\_ACTIVSg70k & 70000 & 83318 & 6 & 88 & 16 & 232\tabularnewline
36 & case74ds & 74 & 73 & 1 & 1 & 3 & 3 & 72 & case\_SyntheticUSA & 82000 & 98203 & 6 & 90 & 17 & 242\tabularnewline
\hline 
\end{tabular}

}
\end{table}
\begin{table}
\caption{\label{tab:opf}Accuracy (in decimal digits) and timing (in seconds)
for 25 largest OPF problems: $n$ - order of matrix variable; \textquotedblleft digit\textquotedblright{}
- accurate decimal digits; \textquotedblleft prep\textquotedblright{}
- all preprocessing time before interior-point, including the conversion
process and MOSEK's internal preparation time; \textquotedblleft iter\textquotedblright{}
- number of interior-point iterations; \textquotedblleft per-it\textquotedblright{}
- interior-point time per iteration; \textquotedblleft post\textquotedblright{}
- post-processing time after interior-point, and to recover $U$ satisfying
$\protect\proj_{F}(UU^{T})=Y$ via \citep[Algorithm 2]{sun2015decomposition}. }

\resizebox{\textwidth}{!}{

\begin{tabular}{|c|cc|cccc|cccc|cccc|c|}
\hline 
 &  &  & \multicolumn{4}{c|}{Dual CTC (provable)} & \multicolumn{4}{c|}{Dual CTC (heuristic)} & \multicolumn{4}{c|}{CC (provable)} & \tabularnewline
\cline{4-15} \cline{5-15} \cline{6-15} \cline{7-15} \cline{8-15} \cline{9-15} \cline{10-15} \cline{11-15} \cline{12-15} \cline{13-15} \cline{14-15} \cline{15-15} 
\# & $n$ & $m$ & prep & digit & iter & per-it & prep & digit & iter & per-it & prep & digit & iter & per-it & post\tabularnewline
\hline 
48 & 1354 & 4060 & 2.14 & 8.4 & 44 & 0.84 & 1.49 & 8.4 & 42 & 0.12 & 0.32 & 8.3 & 47 & 0.05 & 0.24\tabularnewline
49 & 1888 & 5662 & 3.40 & 7.8 & 49 & 1.98 & 2.15 & 7.8 & 53 & 0.15 & 0.41 & 8.0 & 58 & 0.07 & 0.33\tabularnewline
50 & 1951 & 5851 & 3.78 & 7.9 & 46 & 3.17 & 2.21 & 7.7 & 52 & 0.15 & 0.43 & 8.0 & 59 & 0.07 & 0.33\tabularnewline
51 & 2000 & 5998 & 34.68 & 8.1 & 32 & 122.93 & 4.20 & 8.3 & 32 & 2.54 & 1.89 & 7.9 & 28 & 1.23 & 0.98\tabularnewline
52 & 2383 & 7147 & 7.97 & 7.6 & 47 & 14.46 & 2.41 & 7.7 & 41 & 0.32 & 0.81 & 7.7 & 44 & 0.24 & 0.62\tabularnewline
53 & 2736 & 8206 & 9.33 & 7.1 & 57 & 20.59 & 2.83 & 7.8 & 58 & 0.36 & 0.86 & 7.9 & 57 & 0.28 & 0.77\tabularnewline
54 & 2737 & 8209 & 12.31 & 6.7 & 61 & 17.16 & 2.87 & 7.4 & 61 & 0.41 & 0.85 & 8.4 & 56 & 0.31 & 0.76\tabularnewline
55 & 2746 & 8236 & 9.14 & 7.2 & 50 & 18.34 & 4.21 & 7.5 & 51 & 0.58 & 0.92 & 8.4 & 52 & 0.32 & 0.71\tabularnewline
56 & 2746 & 8236 & 10.98 & 7.2 & 60 & 19.83 & 4.07 & 7.2 & 54 & 0.66 & 0.93 & 7.8 & 59 & 0.38 & 1.17\tabularnewline
57 & 2848 & 8542 & 7.52 & 7.5 & 54 & 5.25 & 3.62 & 7.9 & 60 & 0.26 & 0.69 & 6.9 & 58 & 0.18 & 0.57\tabularnewline
58 & 2868 & 8602 & 5.81 & 7.6 & 54 & 6.63 & 3.64 & 7.8 & 60 & 0.23 & 0.70 & 7.8 & 55 & 0.15 & 0.56\tabularnewline
59 & 2869 & 8605 & 9.62 & 8.0 & 47 & 7.36 & 2.98 & 8.0 & 47 & 0.23 & 0.81 & 8.0 & 46 & 0.20 & 0.54\tabularnewline
60 & 3012 & 9034 & 9.85 & 7.7 & 50 & 20.65 & 3.22 & 7.7 & 49 & 0.49 & 1.07 & 8.0 & 52 & 0.37 & 0.95\tabularnewline
61 & 3120 & 9358 & 11.57 & 7.1 & 62 & 21.13 & 3.43 & 7.6 & 61 & 0.54 & 1.12 & 8.2 & 66 & 0.42 & 1.16\tabularnewline
62 & 3374 & 10120 & 12.11 & 7.3 & 57 & 24.02 & 3.85 & 7.7 & 55 & 0.80 & 1.25 & 8.0 & 54 & 0.52 & 1.12\tabularnewline
63 & 6468 & 19402 & 21.38 & 7.7 & 65 & 43.95 & 10.75 & 8.2 & 61 & 1.17 & 1.83 & 7.9 & 63 & 0.73 & 2.37\tabularnewline
64 & 6470 & 19408 & 21.28 & 7.9 & 61 & 40.90 & 8.25 & 8.1 & 59 & 0.79 & 1.92 & 7.8 & 60 & 0.74 & 2.18\tabularnewline
65 & 6495 & 19483 & 20.48 & 7.7 & 68 & 42.03 & 7.95 & 8.1 & 67 & 0.88 & 1.92 & 8.1 & 60 & 0.64 & 2.27\tabularnewline
66 & 6515 & 19543 & 20.48 & 7.6 & 65 & 44.33 & 11.07 & 8.0 & 61 & 1.16 & 1.82 & 7.9 & 60 & 0.62 & 2.16\tabularnewline
67 & 9241 & 27721 & 48.58 & 7.8 & 67 & 120.46 & 14.43 & 7.9 & 54 & 3.02 & 3.69 & 7.8 & 63 & 1.68 & 3.90\tabularnewline
68 & 10000 & 29998 & 55.69 & 8.0 & 49 & 188.19 & 19.83 & 8.2 & 41 & 3.63 & 4.58 & 8.1 & 42 & 2.56 & 4.68\tabularnewline
69 & 13659 & 40975 & 56.28 & 7.2 & 57 & 126.93 & 28.78 & 7.7 & 49 & 3.07 & 4.34 & 7.8 & 50 & 1.84 & 6.37\tabularnewline
70 & 25000 & 74998 & - & - & - & - & 69.16 & 7.4 & 114 & 25.05 & 19.36 & 7.7 & 118 & 14.73 & 25.06\tabularnewline
71 & 70000 & 209998 & - & - & - & - & - & - & - & - & 96.85 & 7.9 & 65 & 180.82 & 113.19\tabularnewline
72 & 82000 & 245994 & - & - & - & - & - & - & - & - & 99.38 & 8.0 & 68 & 197.34 & 140.42\tabularnewline
\hline 
\end{tabular}

}
\end{table}
\begin{figure}
\subfloat[]{\includegraphics[width=0.49\columnwidth]{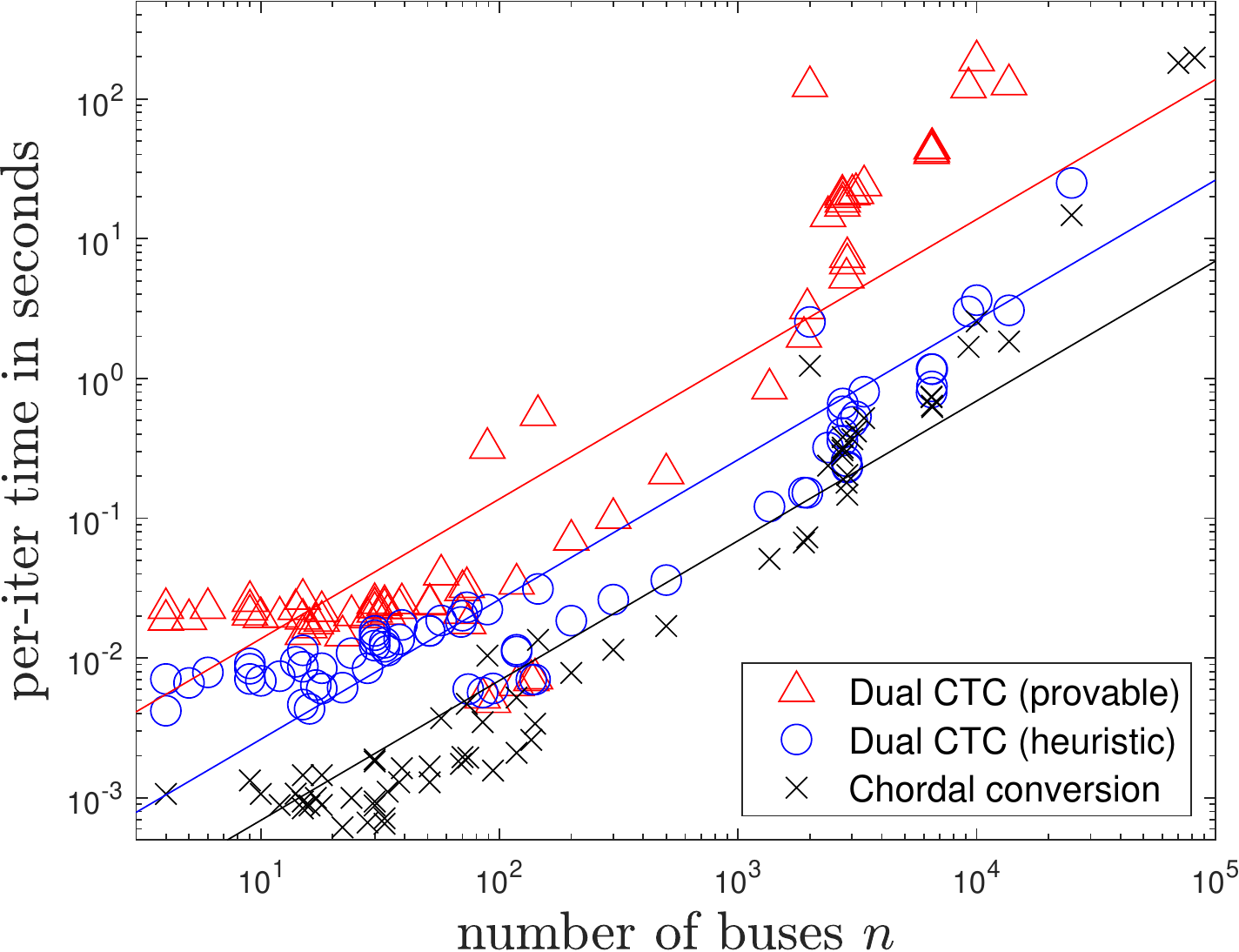}

}\subfloat[]{\includegraphics[width=0.49\columnwidth]{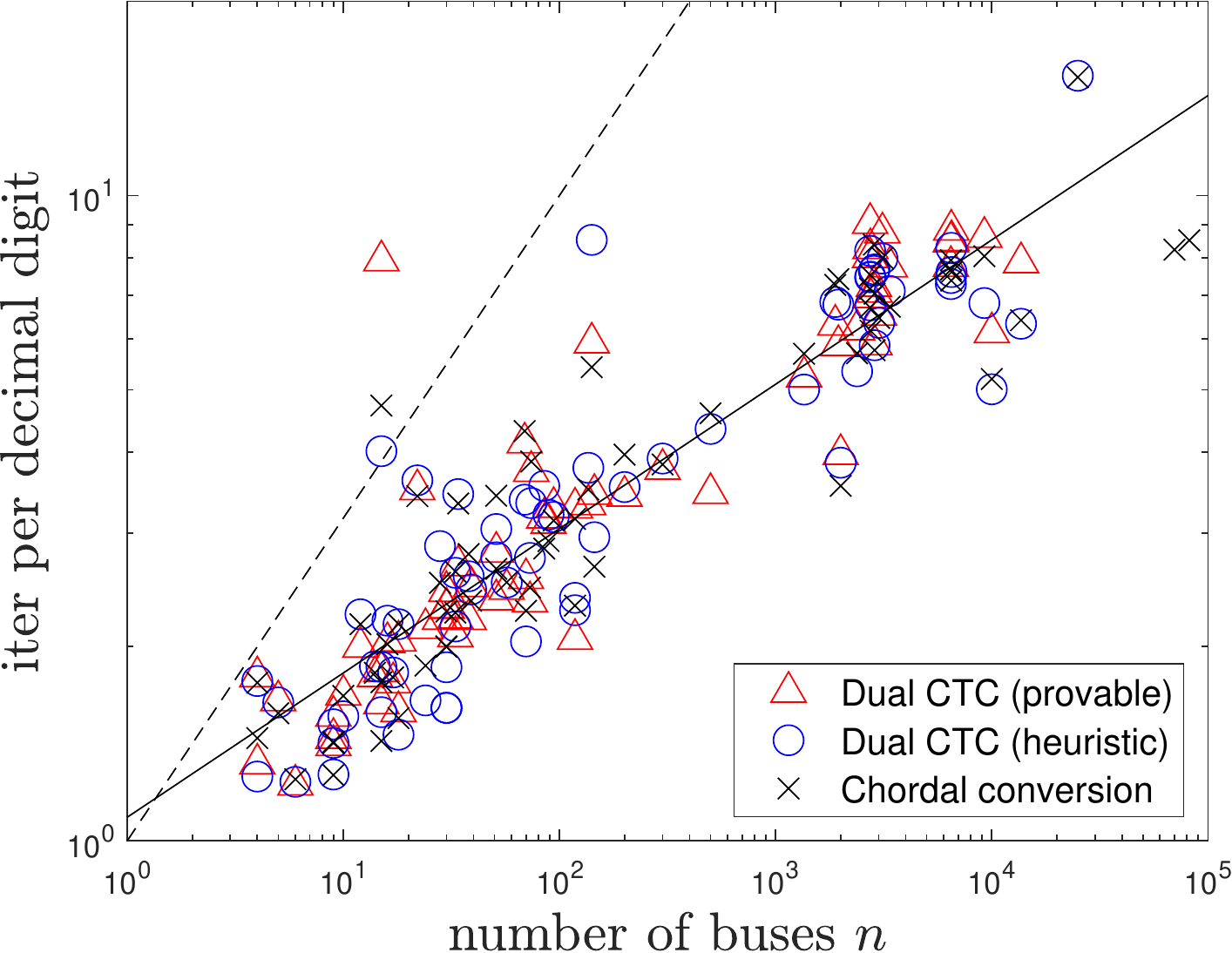}

}

\caption{\label{fig:acopf}AC optimal power flow relaxation solved via chordal
conversion ($\times$), heuristic Dual CTC ($\circ$), and provable
Dual CTC ($\triangle$): (a) Time per iteration, with regression lines
of $f_{\times}(n)=6.908n\times10^{-5}$ with $R^{2}=0.90$, $f_{\circ}(n)=2.623n\times10^{-4}$
with $R^{2}=0.83$, and $f_{\triangle}(n)=1.375n\times10^{-3}$ with
$R^{2}=0.80$; (b) Iterations per decimal digit of accuracy, with
(solid) regression $g(n)=1.088n^{0.224}$ with $R^{2}=0.84$ and (dashed)
bound $g(n)=\sqrt{n}$.}
\end{figure}

Our second set of experiments is on the AC optimal power flow relaxation
(\exaref{acopf}), for which $G\subset\o G$ generally holds with
strict inequality. Here, recall that $\tw(G)=2\cdot\tw(\G)$ and $\tw(\o G)=2\cdot\tw(\mathcal{G}^{2})$,
where $\G$ is the graph of the underlying electric grid, and $\G^{2}$
is its square graph. For the 72 power system graphs taken from the
MATPOWER suite~\citep{zimmerman2011matpower}, we compute upper-
and lower-bounds on $\tw(\G)$ and $\tw(\G^{2})$ using the ``FillIn''
and the ``MMD+'' heuristics outlined in~\citep{maniu2019experimental},
and find that $\tw(\G^{2})$ is small in all 72 power system graphs
(see \tabref{treedecomp}). 

We solve the problem using CC and the two variants of CTC. In all
three cases, it takes a consistent 50 to 70 iterations to achieve
$\epsilon\approx10^{-6}$, again until numerical issues at very large
scales $n\approx10^{4}$ forced more iterations to be taken (see \tabref{opf}).
For these smaller large-scale problems with $G\ne\o G$, we find that
CC and CTC had comparable processing times, but CC is between 2 and
100 times faster than either variant of CTC in its solution time.
The largest case is \texttt{case\_SyntheticUSA} with $82000$ buses
(due to \citep{birchfield2016grid}), which we solve with CC in 4
hours. Previously, this was solved using CTC in 8 hours on a high-performance
computing (HPC) node with two Intel XeonE5-2650v4 processors (a total
of 24 cores) and 240 GB memory~\citep{Eltved2019}.

\section{Conclusions and future directions}

Chordal conversion can sometimes allow an interior-point method to
solve a large-scale sparse SDP in just $O(m+n)$ time per-iteration.
Previously, a well-known necessary condition is that the aggregate
sparsity graph $G=(V,E)$ should have an $O(1)$ treewidth (independent
of $m$ and $n$):
\begin{gather*}
V=\{1,2,\dots,n\},\quad E=\spar(C)\cup\spar(A_{1})\cup\cdots\cup\spar(A_{m}),\\
\text{where }\spar(C)\equiv\{(i,j):C[i,j]\ne0\text{ for }i>j\}.
\end{gather*}
In this paper, provide a companion sufficient condition, namely that
the \emph{extended} aggregate sparsity graph $\o G=(V,\overline{E})$
should also have an $O(1)$ treewidth:
\begin{gather*}
V=\{1,2,\dots,n\},\quad\overline{E}=\spar(C)\cup\clique(A_{1})\cup\cdots\cup\clique(A_{m})\\
\text{where }\clique(A)=\{(i,j):A[i,k]\ne0\text{ or }A[k,j]\ne0\text{ for some }k\}.
\end{gather*}
The key to our analysis is to characterize the Schur complement sparsity
$E^{(2)}$, the sparsity pattern of the linear equations solved at
each iteration, directly in terms of $\o E$. 

Our primary focus has been on reducing the per-iteration costs to
$O(m+n)$. Naively applying this figure to the $O(\sqrt{m+n}\log(1/\epsilon))$
iterations of a general-purpose interior-point method results in an
end-to-end complexity of $O((m+n)^{1.5}\log(1/\epsilon))$ time. By
adopting the treewidth-based interior-point method of Dong, Lee, and
Ye~\citep[Theorem~1.3]{dong2021nearly}, as was done in the recent
preprint of Gu and Song~\citep{gu2022faster}, it should be possible
to formally reduce the end-to-end complexity down to $O((m+n)\log(1/\epsilon))$
time. However, we mention that interior-point methods in practice
often converge to $\epsilon$ accuracy in $O(\log(1/\epsilon))$ iterations
(without the square-root factor), and as such the empirical complexity
is already $O((m+n)\log(1/\epsilon))$ time. 

In many applications, $G$ and $\o G$ either coincide or are very
close, so our analysis becomes either exact or nearly exact; an $O(1)$
treewidth in $\o G$ is both necessary and sufficient for chordal
conversion to be fast. In cases where $G$ and $\o G$ are very different,
particularly when the treewidth of $\o G$ is $\Omega(n)$ while the
treewidth of $G$ is $O(1)$, our preliminary simulations suggest
that the per-iteration cost could slow down $\Omega(n^{3})$ time,
but more work is needed to establish this rigorously. Finally, even
where interior-point methods are no longer efficient, it may still
be possible to use chordal conversion to improve the efficiency of
first-order methods and/or nonconvex approaches. 

\subsubsection*{Acknowledgments}

I am grateful to Martin S. Andersen for numerous insightful discussions,
and for his early numerical experiments that motivated much of the
subsequent theoretical analysis in this paper. The paper has also
benefited significantly from discussions with Subhonmesh Bose, Salar
Fattahi, Alejandro Dominguez--Garcia, Cedric Josz, and Sabrina Zielinski.
I thank the associate editor and two reviewers for helpful comments
that significantly sharpened the presentation of the paper.

\subsubsection*{Funding}

Financial support for this work was provided in part by the NSF CAREER
Award ECCS-2047462 and in part by C3.ai Inc. and the Microsoft Corporation
via the C3.ai Digital Transformation Institute.

\subsubsection*{Conflicts of interest/Competing interests}

The author has no relevant financial or non-financial interests to
disclose.

\appendix

\section{\label{app:regularity}Proof of the standard-form assumptions}

Given data $C,A_{1},\dots,A_{m}\in\S^{n}$ and $b\in\R^{m}$, define
$(\A,\b,\c,\K)$ as in (\ref{eq:AKdef}). In this section, we verify
that $(\A,\b,\c,\K)$ specifies an SDP that satisfies the regularity
assumptions in \defref{clp}. 
\begin{lemma}[Linear independence]
We have $\A^{T}y=0$ if and only $y=0$. 
\end{lemma}
\begin{proof}
We will prove $\sum_{k=1}^{n}\P_{k}\P_{k}^{T}\succeq I$, which implies
$\A\A^{T}\succeq I$ and hence $\|\A^{T}y\|\ge\|y\|$. For arbitrary
$Y\in\S_{F}^{n}$ with $y=\svec_{F}(Y)$, we observe that 
\[
\|\P_{k}^{T}y\|^{2}=\|\svec(Y[J_{k},J_{k}])\|^{2}=\sum_{i,j\in J_{k}}(Y[i,j])^{2}=\sum_{(i,j)\in\clique(J_{k})}\gamma_{ij}(Y[i,j])^{2}
\]
where $\gamma_{ij}=1$ if $i=j$ and $\gamma_{ij}=2$ if $i\ne j$.
Therefore, we have
\[
\sum_{k=1}^{n}\|\P_{k}^{T}y\|^{2}=\sum_{k=1}^{n}\sum_{(i,j)\in\clique(J_{k})}\gamma_{ij}(Y[i,j])^{2}\ge\sum_{(i,j)\in F}\gamma_{ij}(Y[i,j])^{2}=\|y\|^{2}.
\]
The inequality is because $F=\bigcup_{k=1}^{n}\clique(J_{k})$. \qed
\end{proof}

\begin{lemma}[Strong duality is attained]
Under \asmref{strict}, there exists $x^{\star},s^{\star}\in\K$
satisfying $\A x^{\star}=\b,$ $\A^{T}y^{\star}+s^{\star}=\c,$ $(x^{\star})^{T}s^{\star}=0$. 
\end{lemma}
\begin{proof}
Define $P_{j}$ implicitly to satisfy $P_{j}^{T}x=x[\col_{F}(j)]$
for all $x\in\R^{m}$. \asmref{strict} says that there exists $\hat{X}\succeq0$
and $\hat{y}\le0$ and that satisfy $\AA(\hat{X})\le b$ and $\AA^{T}(\hat{y})+\hat{S}=C$
and $\inner C{\hat{X}}=\inner b{\hat{y}}$. It follows from \citep[Theorem~9.2]{vandenberghe2015chordal}
that $\hat{S}=C-\AA^{T}(\hat{y})\in\S_{F}^{n}\cap\S_{+}^{n}$ if and
only if there exists $\hat{S}_{j}\succeq0$ such that $\hat{S}=\sum_{j=1}^{n}P_{j}\hat{S}_{j}P_{j}^{T}$.
Now, we turn to the primal-dual pair defined by the data $(\A,\b,\c,\K)$
in (\ref{eq:AKdef}), which is written 
\begin{align*}
 & \min_{v\in\R^{m},V_{j}\in\S^{\omega_{j}}}\left\{ \inner bv:\begin{array}{c}
\proj_{F}\left[\AA^{T}(-v)+\sum_{j=1}^{n}P_{j}V_{j}P_{j}^{T}-C\right]=0,\\
v\ge0,\quad V_{1},\dots,V_{n}\succeq0
\end{array}\right\} \\
= & \max_{Y\in\S_{F}^{n}}\left\{ \inner{-C}Y:\AA(Y)\le b,\;-P_{j}^{T}YP_{j}\preceq0\text{ for all }j\in\{1,2,\dots,n\}\right\} 
\end{align*}
We can mechanically verify that $v^{\star}=-\hat{y}$ and $V_{j}^{\star}=\hat{S}_{j}$
is feasible for the primal problem, and $Y^{\star}=\Pi_{F}(\hat{X})$
is feasible for the dual problem, and that the two objectives coincide
$\inner b{v^{\star}}=\inner{-C}{Y^{\star}}$. Therefore, we conclude
that $x^{\star}=(v^{\star},\svec(V_{1}^{\star}),\dots,\svec(V_{n}^{\star}))$
and $y^{\star}=\svec_{F}(Y^{\star})$ is a complementary solution
satisfying $\A^{T}x^{\star}=\b,$ $\A^{T}y^{\star}+s^{\star}=\c,$
and $\inner{x^{\star}}{s^{\star}}=0$. \qed 
\end{proof}

\section{\label{app:ipm}Complexity of general-purpose interior-point methods}

In this section, we state a concrete interior-point method that implements
the specifications outlined in \defref{ipm}, roughly following the
steps in~\citep{de2000self}.
\begin{proposition}
\label{prop:general}Let $(\A,\b,\c,\K)$ describe an SDP in $(n,\omega)$-standard
form, and let $\opt$ denote its optimal value. Let $\Tim$ and $\Mem$
denote the time and memory needed, given data $\A\in\R^{M\times N}$,
$g\in\R^{M}$, and a choice of $w\in\interior(\K)$, to form and solve
$\A\D\A^{T}\Delta y=g$ for $\Delta y\in\R^{M}$, where $\D^{-1}=\nabla^{2}f(w)$
and $f=-\log\det(w)$. Then, a general-purpose interior-point method
computes $(x,y)$ that satisfy the following in $O(\sqrt{n}\log(1/\epsilon))$
iterations
\begin{gather*}
\c^{T}x\le\mathrm{opt}+n\cdot\epsilon,\quad\|\A x-\b\|\le\epsilon,\quad x\in\K,\\
\b^{T}y\ge\mathrm{opt}-n\cdot\epsilon,\quad\c-\A^{T}y+\epsilon\one_{\K}\in\K,
\end{gather*}
with per-iteration costs of $O(\omega^{2}n+\nnz(\A)+\Tim)$ time and
$O(\omega n+\nnz(\A)+\Mem)$ memory.
\end{proposition}
We prove \propref{general} by using the short-step method of Nesterov
and Todd~\citep[Algorithm~6.1]{nesterov1998primal} to solve the
extended homogeneous self-dual embedding\begin{subequations}\label{eq:hsd}
\begin{alignat}{2}
 & \min_{x,y,s,\kappa,\tau,\theta}\quad & (n+1)\theta\\
 & \text{s.t.} & \begin{bmatrix}0 & +\A^{T} & -\c & -r_{d}\\
-\A & 0 & +\b & -r_{p}\\
+\c^{T} & -\b^{T} & 0 & -r_{c}\\
r_{d}^{T} & r_{p}^{T} & r_{c} & 0
\end{bmatrix}\begin{bmatrix}x\\
y\\
\tau\\
\theta
\end{bmatrix}+\begin{bmatrix}s\\
0\\
\kappa\\
0
\end{bmatrix} & =\begin{bmatrix}0\\
0\\
0\\
n+1
\end{bmatrix}\label{eq:hsdfeas}\\
 &  & x,s\in\K,\quad\kappa,\tau & \ge0,\label{eq:hsdcone}
\end{alignat}
\end{subequations}where $r_{d}=\one_{\K}-\c$ and $r_{p}=-\A\one_{\K}+\b$
and $r_{c}=1+\c^{T}\one_{\K}$. Beginning at the strictly feasible,
perfectly centered point in (\ref{eq:hsdinit}) for $\mu=1$:
\begin{gather}
\theta^{(0)}=\tau^{(0)}=\kappa^{(0)}=1,\quad y^{(0)}=0,\quad x^{(0)}=s^{(0)}=\one_{\K}.\label{eq:hsdinit}
\end{gather}
 we take the following steps\begin{subequations}\label{eq:hsd_ipm}
\begin{align}
\mu^{+} & =\left(1-\frac{1}{15\sqrt{n+1}}\right)\cdot\frac{x^{T}s+\tau\kappa}{n+1},\\
(x^{+},y^{+},s^{+},\tau^{+},\theta^{+},\kappa^{+}) & =(x,y,s,\tau,\theta,\kappa)+(\Delta x,\Delta y,\Delta s,\Delta\tau,\Delta\theta,\Delta\kappa).\nonumber 
\end{align}
\end{subequations}along the search direction defined by the linear
system~\citep[Eqn. 9]{todd1998nesterov} \begin{subequations}\label{eq:hsd_nt}
\begin{align}
\begin{bmatrix}0 & +\A^{T} & -\c & -r_{p}\\
-\A & 0 & +\b & -r_{d}\\
+\c^{T} & -\b^{T} & 0 & -r_{c}\\
+r_{p}^{T} & +r_{d}^{T} & +r_{c} & 0
\end{bmatrix}\begin{bmatrix}\Delta x\\
\Delta y\\
\Delta\tau\\
\Delta\theta
\end{bmatrix}+\begin{bmatrix}\Delta s\\
0\\
\Delta\kappa\\
0
\end{bmatrix} & =0,\\
s+\Delta s+\nabla^{2}f(w)\Delta x+\mu^{+}\nabla f(x) & =0,\\
\kappa+\Delta\kappa+(\kappa/\tau)\Delta\tau-\mu^{+}\tau^{-1} & =0.
\end{align}
\end{subequations}where $f(w)=-\log\det w$ is the usual self-concordant
barrier function, and $w\in\interior(\K)$ is the unique point that
satisfies $\nabla^{2}f(w)x=s$. The following is an immediate consequence
of~\citep[Theorem~6.4]{nesterov1998primal}.
\begin{lemma}[Short-Step Method]
\label{lem:short-step}The sequence in (\ref{eq:hsd_ipm}) arrives
at an iterate $(x,y,s,\tau,\theta,\kappa)$ satisfying
\begin{equation}
\frac{x^{T}s+\tau\kappa}{n+1}\le\epsilon,\qquad\tau\kappa\ge\gamma\epsilon\label{eq:short-step}
\end{equation}
with $\gamma=\frac{9}{10}$ in at most $O(\sqrt{n}\log(1/\epsilon))$
iterations. 
\end{lemma}
Finally, the following result (adapted from~\citep[Lemma~5.7.2]{de2000self})
assures us that a feasible point satisfying the optimality condition
(\ref{eq:short-step}) will recover a solution to the original problem.

\begin{lemma}[$\epsilon$-accurate and $\epsilon$-feasible]
\label{lem:eps}Let there exist $(x^{\star},y^{\star})$ such that
$\A x^{\star}=\b$ with $x^{\star}\in\K$ and $\c-\A^{T}y^{\star}\equiv s^{\star}\in\K$
that satisfies strong duality $\b^{T}y^{\star}=\c^{T}x^{\star}$.
Then, a point $(x,y,s,\tau,\theta,\kappa)$ that is feasible for (\ref{eq:hsd})
and satisfies the optimality condition (\ref{eq:short-step}) also
satisfies the following, where $\rho=1+\one_{\K}^{T}(x^{\star}+s^{\star})$:
\begin{align*}
\|\A(x/\tau)-\b\| & \le\frac{\rho\|r_{p}\|}{\gamma}\cdot\epsilon, & \|\A^{T}(y/\tau)+(s/\tau)-\c\| & \le\frac{\rho\|r_{d}\|}{\gamma}\cdot\epsilon, & \frac{(x/\tau)^{T}(s/\tau)}{n+1}\le\frac{\rho^{2}}{\gamma^{2}}\cdot\epsilon.
\end{align*}
\end{lemma}
\begin{proof}
First, observe that $(\hat{x},\hat{y},\hat{s},\hat{\tau},\hat{\theta},\hat{\kappa})=(\hat{\tau}x^{\star},\hat{\tau}y^{\star},\hat{\tau}s^{\star},\hat{\tau},0,0)$
with $\hat{\tau}=(n+1)/\rho$ is a solution to (\ref{eq:hsd}), because
$r_{d}^{T}x^{\star}+r_{p}^{T}y^{\star}+r_{c}=\one_{\K}^{T}(x^{\star}+s^{\star})+1=\rho$.
Next, we prove that if $(x,y,s,\tau,\theta,\kappa)$ is feasible for
(\ref{eq:hsd}) and satisfies (\ref{eq:short-step}), then $\tau\ge\frac{\gamma}{\rho}$.
Indeed, the skew-symmetry of (\ref{eq:hsdfeas}) yields $\theta=\frac{x^{T}s+\tau\kappa}{n+1}$
and $(x-\hat{x})^{T}(s-\hat{s})+(\tau-\hat{\tau})(\kappa-\hat{\kappa})=0.$
Rearranging yields $(n+1)\theta=x^{T}s+\tau\kappa=\hat{x}^{T}s+x^{T}\hat{s}+\tau\hat{\kappa}+\hat{\tau}\kappa$
and hence $\kappa\le(n+1)\theta/\hat{\tau}$. Under (\ref{eq:short-step}),
we have $\theta=\frac{x^{T}s+\tau\kappa}{n+1}\le\epsilon$ and $\tau\ge\frac{\gamma\epsilon}{\kappa}\ge\frac{\gamma\theta}{\kappa}\ge\frac{\gamma\theta}{(n+1)\theta/\hat{\tau}}=\frac{\gamma}{n+1}\hat{\tau}=\frac{\gamma}{\rho}$.
Finally, divide (\ref{eq:hsdfeas}) through by $\tau$ and substitute
$1/\tau\le\rho/\gamma$. \qed
\end{proof}

\begin{proof}[Proof of \propref{general}]Recall that $(\A,\b,\c,\K)$
describe an SDP in $(n,\omega)$-standard form, and hence there exist
$(x^{\star},y^{\star})$ such that $\A x^{\star}=\b$ with $x^{\star}\in\K$
and $\c-\A^{T}y^{\star}\in\K$ that satisfies strong duality $\b^{T}y^{\star}=\c^{T}x^{\star}$.
Combining \lemref{eps} and \lemref{short-step} shows that iterations
in (\ref{eq:hsd_ipm}) converges to the desired $\epsilon$-accurate,
$\epsilon$-feasible iterate after $O(\sqrt{n}\log(1/\epsilon))$
iterations. The cost of each iteration is essentially equal to the
cost of computing the search direction in (\ref{eq:hsd_nt}). We account
for this cost via the following two steps:
\begin{enumerate}
\item (Scaling point) We partition $x=[\svec(X_{j})]_{j=1}^{\ell}$ and
$s=[\svec(S_{j})]_{j=1}^{\ell}$. Then, the scaling point $w=[\svec(W_{j})]_{j=1}^{\ell}$
is given in closed-form as $W_{j}=S_{j}^{1/2}(S_{j}^{1/2}X_{j}S_{j}^{1/2})^{-1/2}S_{j}^{1/2}$~\citep[Section~5]{sturm2002implementation}.
Noting that each $W_{j}$ is at most $\omega\times\omega$, the cost
of forming $w$ is at most of order 
\[
\sum_{j=1}^{\ell}\omega_{j}^{3}\le\omega^{2}\sum_{j=1}^{\ell}\omega_{j}=\omega^{2}n\text{ time},\quad\sum_{j=1}^{\ell}\omega_{j}^{2}\le\omega\sum_{j=1}^{\ell}\omega_{j}=\omega n\text{ memory.}
\]
By this same calculation, it follows that the Hessian matrix-vector
products $\nabla^{2}f(w)x=[\svec(W_{j}^{-1}X_{j}W_{j}^{-1})]_{j=1}^{\ell}$
and $\nabla^{2}f(w)^{-1}x=[\svec(W_{j}X_{j}W_{j})]_{j=1}^{\ell}$
also cost $O(\omega^{2}n)$ time and $O(\omega n)$ memory. 
\item (Search direction) Using elementary but tedious linear algebra, we
can show that if\begin{subequations}\label{eq:newt}
\begin{equation}
(\A\D\A^{T})\begin{bmatrix}v_{1} & v_{2} & v_{3}\end{bmatrix}=\begin{bmatrix}0 & -\b & r_{p}\end{bmatrix}-\A\D\begin{bmatrix}d & \c & r_{d}\end{bmatrix}\label{eq:hsd_nrm}
\end{equation}
where $\D=[\nabla^{2}f(w)]^{-1}$ and $d=-s-\mu^{+}\nabla f(x)$,
and 
\begin{equation}
\begin{bmatrix}u_{1} & u_{2} & u_{3}\end{bmatrix}=\D(\begin{bmatrix}d & c & r_{d}\end{bmatrix}+\A^{T}\begin{bmatrix}v_{1} & v_{2} & v_{3}\end{bmatrix}),\label{eq:hsd_back1}
\end{equation}
then
\begin{align}
\left(\begin{bmatrix}-\D_{0}^{-1} & -r_{c}\\
r_{c} & 0
\end{bmatrix}-\begin{bmatrix}\c & r_{d}\\
-\b & r_{p}
\end{bmatrix}^{T}\begin{bmatrix}u_{2} & u_{3}\\
v_{2} & v_{3}
\end{bmatrix}\right)\begin{bmatrix}\Delta\tau\\
\Delta\theta
\end{bmatrix} & =\begin{bmatrix}-d_{0}\\
0
\end{bmatrix}-\begin{bmatrix}\c & r_{d}\\
-\b & r_{p}
\end{bmatrix}^{T}\begin{bmatrix}u_{1}\\
v_{1}
\end{bmatrix},\\
\begin{bmatrix}\Delta x\\
\Delta y
\end{bmatrix} & =\begin{bmatrix}u_{1}\\
v_{1}
\end{bmatrix}-\begin{bmatrix}u_{1} & u_{2}\\
v_{1} & v_{2}
\end{bmatrix}\begin{bmatrix}\Delta\tau\\
\Delta\theta
\end{bmatrix},\\
\Delta s & =d-\D^{-1}\Delta x,\\
\Delta\kappa & =d_{0}-\D_{0}^{-1}\Delta\tau,\label{eq:hsd_back2}
\end{align}
\end{subequations}where $\D_{0}=\tau/\kappa$ and $d_{0}=-\kappa+\mu^{+}\tau^{-1}$.
Therefore, we conclude that the cost of computing the search direction
in (\ref{eq:hsd_nt}) is equal to the cost of solving $O(1)$ instances
of the Schur complement equation $\A\D\A^{T}\Delta y=g$, plus $O(1)$
matrix-vector products with $\A,\A^{T},\D,\D^{-1}$, for a total cost
of $O(\Tim+\nnz(\A)+\omega^{2}n)$ time and $O(\Mem+\nnz(\A)+\omega n)$
memory respectively. Note that $\A\D\A^{T}\succ0$ because $\A\A^{T}\succ0$
by the linear independence assumption, and $\D^{-1}=\nabla^{2}f(w)\succ0$
for all $w\in\interior(\K)$.
\end{enumerate}
\qed \end{proof}

\bibliographystyle{abbrvnat}
\bibliography{chordConv_short}

\end{document}